\documentclass{amsart}
\usepackage{macros}

\begin{document}
\title{Orthogonality relations for deep level Deligne--Lusztig schemes of Coxeter type}
\date{\today}

\author{Olivier Dudas}
\address{Universit\'e de Paris and Sorbonne Universit\'e, CNRS, IMJ-PRG, F-75006 Paris, France.}
\email{olivier.dudas@imj-prg.fr}

\author{Alexander B. Ivanov}
 \address{Mathematisches Institut, Universit\"at Bonn, Endenicher Allee 60, 53115 Bonn, Germany.}
 \email{ivanov@math.uni-bonn.de}

\begin{abstract} 
In this paper we prove some orthogonality relations for representations arising from deep level Deligne--Lusztig schemes of Coxeter type. This generalizes previous results of Lusztig \cite{Lusztig_04}, and of Chan and the second author \cite{CI_MPDL}. Potential applications include the study of unipotent representations arising from such deep level Deligne--Lusztig schemes, as well as their geometry, in the spirit of Lusztig's work \cite{Lusztig_76_Inv}.
\end{abstract}

\maketitle

\section{Introduction}

In the last fifteen years various $p$-adic and deep level analoga of classical Deligne--Lusztig varieties attracted a lot of attention, see in particular \cite{Lusztig_04, Boyarchenko_12, ChenS_17, Chan_siDL, CI_loopGLn, Ivanov_DL_indrep}. The interest in them is justified by the fact that they allow to apply methods from classical Deligne--Lusztig theory to study representations of $p$-adic groups. Furthermore, they are very interesting geometric objects in their own right (like the classical Deligne--Lusztig varieties are). In this article we consider deep level Deligne--Lusztig schemes of Coxeter type, and prove orthogonality relations for the corresponding representations, extending a classical result of \cite{DeligneL_76} to the deep level setup.

Let $k$ be a non-archimedean local field with uniformizer $\varpi$ and residue field $\bF_q$ with $q$ elements. Let $\breve k$ denote the completion of a maximal unramified extension of $k$ with residue field $\obF$. Let $\bfG$ be an unramified reductive group over $k$, $\bfT \subseteq \bfG$ a $k$-rational unramified maximal torus, and $\bfU$ the unipotent radical of a $\breve k$-rational Borel subgroup of $\bfG$ containing $\bfT$. Let $\caG = \caG_{\bf x}$ be a (connected) parahoric $\caO_k$-group scheme with generic fiber $\bfG$, whose corresponding facet ${\bf x}$ in the Bruhat--Tits building (over $k$) of the adjoint group of $\bfG$ lies in the apartment of $\bfT$, and let $\caT \subseteq \caG$ denote the schematic closure of $\bfT$ in $\caG$.

Fix an integer $r \geq 1$. Let $G = \caG(\caO_k/\varpi^r)$ and $T = \caT(\caO_k/\varpi^r)$. In \cite{Lusztig_04, Stasinski_09, CI_MPDL}, a certain (perfect) $\obF$-scheme $S_{\bfT,\bfU} = S_{{\bf x},\bfT,\bfU,r}$ equipped with a natural $G\times T$-action was defined. In a sense, it can be regarded as a deep level analog of a classical Deligne--Lusztig variety\footnote{For example, if $r=1$ and $\bf x$ hyperspecial, then there is a natural map from $S_{\bfT,\bfU}$ into some classical Deligne--Lusztig variety attached to the special fibers of $\caG,\caT,\caU$; this map induces an isomorphism of $\ell$-adic cohomology groups, up to a degree shift.} As in the classical Deligne--Lusztig theory, the cohomology of $S_{\bfT,\bfU}$ attaches to any character $\theta \colon T \rar \overline \bQ_\ell^\times$ ($\ell \neq \charac \bF_q$) the $G$-representation $R_{\bfT,\bfU}(\theta) = \sum_{i \in \bZ} (-1)^i H_c^i(S_{\bfT,\bfU},\overline \bQ_\ell)_\theta$.  One of the central features within the classical Deligne--Lusztig theory is the \emph{Deligne--Lusztig orthogonality relation}, which computes (in the classical case, that is $r=1$, $\bf x$ hyperspecial) the inner product of two virtual representations $R_{\bfT,\bfU}(\theta)$, $R_{\bfT',\bfU'}(\theta')$ \cite[Thm. 6.8]{DeligneL_76}.

The goal of the present article is to generalize the abovementioned classical orthogonality relations to deep level schemes $S_{\bfT,\bfU}$ of Coxeter type. There is a meaningful notion of a \emph{Coxeter pair} $(\bfT,\bfU)$ (cf. Section \ref{sec:Cox_tori_setup}), which essentially means that $S_{\bfT,\bfU}$ is the deep level analog of a classical Deligne--Lusztig variety of Coxeter type. In that case the intersection of the apartment of $\bfT$ in the $k$-rational Bruhat--Tits building of the adjoint group is just one vertex, ${\bf x}_\bfT$, and (as we assumed $\bfG$ to be unramified) this vertex must necessarily be hyperspecial (cf. Section \ref{sec:Cox_tori_setup}). The following theorem is our main result.

\begin{repthm}{thm:higher_mackey}
Let $(\bfT,\bfU)$, $(\bfT',\bfU')$ be Coxeter pairs with $\bf x = {\bf x}_{\bfT} = {\bf x}_{\bfT'}$ (then, automatically, ${\bf x}$ is hyperspecial). Assume that $q > 5$. Then for all $r \geq 1$ and all $\theta \colon T \rar \overline \bQ_\ell^\times$, $\theta' \colon T' \rar \overline \bQ_\ell^\times$, we have
\[
\langle R_{\bfT,\bfU}(\theta), R_{\bfT',\bfU'}(\theta') \rangle_G = \# \left\{w \in W(\bfT,\bfT')^F \colon \theta' = {}^w\theta \right\},
\]
where $W(\bfT,\bfT') =\bfT(\breve k) \backslash \{ g \in \bfG(\breve k)\colon {}^g\bfT' = \bfT\}$.
\end{repthm}

Note that the assumption on $q$ can be strengthened depending on the root system of $\bfG$, see Condition~\eqref{cond:condition_on_rootsystem}.

\smallskip 

Our proof follows an idea of \cite{Lusztig_04} (already appearing in \cite{DeligneL_76}) which consists in extending the $T \times T'$-action on various subschemes of $\Sigma = G \backslash(S_{\bfT,\bfU} \times S_{\bfT',\bfU'})$ to an action of some torus with finitely many fixed points. This considerably simplifies the computation of the Euler characteristic of $\Sigma$. For $\bfG = {\bf GL}_n$ (and, essentially, for any unramified group of type $A_n$) Theorem \ref{thm:higher_mackey}  was proven in \cite[\S3]{CI_loopGLn}. However, the general case requires several serious improvements, which are the core of the present work. 

Let us explain the importance of Theorem \ref{thm:higher_mackey}. Under the additional condition that $\theta$ (or $\theta'$) is \emph{regular}, i.e. ``highly non-trivial'' on $\ker(\caT(\caO_k/\varpi^r) \rar \caT(\caO_k/\varpi^{r-1}))$, the result was already shown in \cite{Lusztig_04,Stasinski_09}. However, from the perspective of unipotent representations, the most interesting case is that of $\theta = \theta' = 1$. In the classical situation ($r=1$, $\bf x$ hyperspecial), the characters $R_{\bfT,\bfU}(1)$ as well as the geometry and cohomology of the related Coxeter type Deligne--Lusztig varieties were studied in Lusztig's seminal work \cite{Lusztig_76_Inv}; one of the starting points for \cite{Lusztig_76_Inv} was the orthogonality relation \cite[Thm. 6.8]{DeligneL_76} with $\theta =\theta' = 1$. In light of this, Theorem \ref{thm:higher_mackey} now opens up the way towards the study of deep level schemes $X_w^{\caG,r}(1)$ of Coxeter type (which are isomorphic to certain quotients of $S_{\bfT,\bfU}$) and the corresponding ``unipotent'' representations (=irreducible constituents of $R_{\bfT,\bfU}(1)$) of the group $\caG(\caO_k/\varpi^r)$ and their inflation to $\caG(\caO_k)$.

The schemes $X_w^{\caG,r}(1)$ are integral truncated versions of (and closely related to) the ``big'' $p$-adic Deligne--Lusztig spaces $X_w(1)$, which carry an $\bfG(k)$-action, studied in \cite{Ivanov_DL_indrep}. At least in special cases, the ``$p$-adic Deligne--Lusztig'' $\bfG(k)$-representations in the cohomology of $X_w(1)$ realize interesting supercuspidal representations of $\bfG(k)$, related to the local Langlands correspondences (cf. \cite[Thm. A]{CI_loopGLn}), and they are compactly induced from the $\caG(\caO_k)$-representations $R_{\bfT,\bfU}(\theta)$ considered here.

\subsection*{Acknowledgements} 
The first author gratefully acknowledges financial support by the grant ANR-16-CE40-0010-01. The second author was supported by the DFG through the Leibniz Prize of Peter Scholze. 

\section{Setup and preliminaries} 
\subsection{Some notation}
Given a group $G$ and $g,x \in G$, we write ${}^g x = gxg^{-1}$ and $x^g = g^{-1}xg$. If $\theta$ is an irreducible character of a finite subgroup $H$ of $G$ then ${}^g\theta$ is the character of $H^g$ given by ${}^g\theta(x) := \theta (g x g^{-1})$. 

Let $p$ be a prime number. Given a ring $R$ of characteristic $p$, we denote by $\Perf_R$ the category of perfect $R$-algebras, and by $W(R)$ the ($p$-typical) Witt vectors of $R$.

Let $k$ be a non-archimedean local field with residue field $\bF_q$, where $q$ is some fixed power of $p$. The ring of integers of $k$ will be denoted by $\caO_k$. Let $\varpi$ be a uniformizer of $k$. Given $R \in \Perf_{\bF_q}$, there is an essentially unique $\varpi$-adically complete and separated $\caO_k$-algebra $\bW(R)$, in which $\varpi$ is not a zero-divisor and which satisfies $\bW(R)/ \varpi\bW(R) = R$. Explicitly we have 
\[
\bW(R) = \begin{cases}  W(R) \otimes_{W(\bF_q)} \caO_k &\text{if $\charac k = 0$} \\ R[\![\varpi]\!] &\text{if $\charac k = p$,}\end{cases}
\]
i.e., in the first case $\bW(R)$ are the ramified Witt vectors, details on which can be found for example in \cite[1.2]{FarguesFontaine_book}. In particular, $\bW(\bF_q)[1/\varpi] = k$. Fix an algebraic closure $\obF$  of $\bF_q$ and put $\caO_{\breve k} = \bW(\obF)$ and $\breve k = \bW(\obF)[1/\varpi]$. The field $\breve k$ is the $\varpi$-adic completion of a maximal unramified extension of $k$. 

\subsection{Loop functors}
Let $\caX$ be an $\caO_{\breve k}$-scheme. We have the functor of positive loops and its truncations (also called Greenberg functors, following \cite{Greenberg_61})
\begin{align*} 
L^+\caX \colon \Perf_{\obF} &\rar {\rm Sets}, \quad  (L^+\caX)(R) = \caX(\bW(R)) \\
L_r^+\caX \colon \Perf_{\obF} &\rar {\rm Sets}, \quad  (L_r^+\caX)(R) = \caX(\bW(R)/\varpi^r\bW(R)).
\end{align*}
If $\caX$ is affine of finite type over $\caO_{\breve k}$, then $L^+\caX$ and $L_r^+\caX$ are representable by affine perfect schemes, and the latter is of perfectly finite type over $\obF$, as follows from \cite{Greenberg_61}.

Moreover, if $\caX$ is equipped with an $\caO_k$-rational structure, i.e., $\caX = \caX_0 \otimes_{\caO_k} \caO_{\breve k}$ for an $\caO_k$-scheme $\caX_0$, then $L^+\caX$ and $L^+_r\caX$ both come equipped with geometric Frobenius automorphisms (over $\obF$), which we denote by $F \colon L^+\caX \rar L^+\caX$ resp. $F \colon L^+_r\caX \rar L^+_r\caX$.

\subsection{Perfect schemes and $\ell$-adic cohomology}\label{sec:perfect_schemes_and_coh}

We fix a prime $\ell \neq p$, and an algebraic closure $\overline \bQ_\ell$ of $\bQ_\ell$. Without further reference we will make use of the formalism of \'etale cohomology with compact support, as developed in \cite{Deligne_77}. If $f \colon X \rar \Spec \obF$ is a (separated) morphism of finite type, then we put $H_c^i(X,\overline \bQ_\ell) = R^if_!\overline \bQ_\ell$, where $\overline \bQ_\ell$ is the constant local system of rank $1$ on $X$. Then $H_c^i(X, \overline \bQ_\ell)$ is a finite dimensional $\overline \bQ_\ell$-vector space, which is zero for almost all $i \in \bZ$, and we may form the $\ell$-adic Euler characteristic $H_c^\ast(X) = \sum_{i\in \bZ} (-1)^iH_c^i(X,\overline \bQ_\ell)$ of $X$, which is an element of the Grothendieck group of finite dimensional $\overline \bQ_\ell$-vector spaces. 

If $X$ is a perfect scheme over $\obF$, such that the structure morphism $f \colon X \rar \Spec \obF$ is (separated and) of perfectly finite type, we may choose any model $f_0 \colon X_0 \rar \Spec \obF$ of finite type over $\obF$, such that $f$ is the perfection of $f_0$. Then we have $H_c^\ast(X) = H_c^\ast(X_0)$. Hence, the above cohomological formalism extends to (separated) perfectly finitely presented perfect schemes over $\obF$. In particular, $H_c^\ast(X)$ makes sense as a virtual (finite) $\overline \bQ_\ell$-vector space. Moreover, if $X$ is acted on by a finite group $G$, then by functoriality, $H_c^\ast(X)$ is a virtual $\overline \bQ_\ell G $-module.

Below (in Section \ref{sec:comparison_of_cells}) we will often encounter the following situation. Let $X$ and $G$ be as in the preceding paragraph. Suppose that $X$ is affine and that there is a torus $\bT$ over $\obF$ which acts on $X$, and that this action commutes with the $G$-action. Then 
\begin{equation}\label{eq:Euler_char}
H_c^\ast(X) = H_c^\ast(X^T)  
\end{equation}
as $\overline \bQ_\ell[G]$-modules, as follows from \cite[10.15]{DigneM_91}. This will apply to schemes $\widehat\Sigma_w$, $\widetilde\Sigma_w$ constructed in Sections \ref{sec:Lusztigs_extension},\ref{sec:another_extension_of_action}. We also must apply this to the schemes $Y_{v,w}$ (resp. $Z_{v,w}$) constructed in Section \ref{sec:isomorphism}, which are locally closed subschemes of $\widehat\Sigma_{v,w}$ of which we do not know that they are affine. However, the action of the torus $\bT$ on $Y_{v,w}$ (resp. $Z_{v,w}$) will be the restriction of an action of the same $\bT$ on $\widehat\Sigma_w$ (resp. $\widehat\Sigma_v$). In this situation the proof of \cite[10.15]{DigneM_91} still applies and hence \eqref{eq:Euler_char} still holds for $Y_{v,w}$, $Z_{v,w}$. 

In the rest of this article all schemes over $\bF_q$ or $\obF$ will be separated, perfect and of perfectly finite type (unless specified otherwise). Whenever we consider objects over $\bF_q$ or $\obF$, we simply write ``scheme'' for ``perfect scheme''. 

\subsection{Groups, parahoric models and Moy--Prasad quotients}\label{sec:groups_setup}
Let $\bfG$ be a reductive group over $k$ which splits over $\breve k$. For $E \in \{k,\breve k\}$, let $\caB(\bfG, E)$ be the Bruhat--Tits building of the adjoint group of $\bfG$. The Frobenius of $\breve k/k$ induces automorphisms of $\bfG(\breve k)$ and $\caB(\bfG,\breve k)$, both denoted by $F$, and we have $\bfG(\breve k)^F = \bfG(k)$ and $\caB(\bfG,\breve k)^F = \caB(\bfG,k)$. 

Let ${\rm Tori}_{\breve k/k}(\bfG)$ be the set of $k$-rational $\breve k$-split maximal tori of $\bfG$. Given $\bfT \in {\rm Tori}_{\breve k/k}(\bfG)$, we denote by $X^\ast(\bfT)$ (resp. $X_\ast(\bfT)$) the group of characters (resp. cocharacters) of $\bfT$, and by $\Phi(\bfT,\bfG) \subseteq X^\ast(\bfT)$ the set of roots of $\bfT$ in $\bfG$. Given $\alpha \in \Phi(\bfT,\bfG)$, $\bfU_\alpha$ denotes the corresponding root subgroup. Furthermore, we denote by $F$ the automorphism of $X^\ast(\bfT)$ resp. $X_\ast(\bfT)$ induced by the Frobenius of $\breve k/k$. Let $\caA(\bfT, \breve k)$ denote the apartment of $\bfT$ in $\caB(\bfG,\breve k)$. 

From the theory of Bruhat--Tits, we can attach to any facet ${\bf x} \in \caB(\bfG,k)$  a connected parahoric $\caO_k$-model $\caG_{\bf x}$ of $\bfG$ \cite[\S4.6, 5.2.6]{BruhatT_84}. It is smooth, affine and has generic fiber $\bfG$. 
The group $\caG_{\bf x}(\caO_{\breve k})$ admits a Moy--Prasad filtration by subgroups $\caG_{\bf x}(\caO_{\breve k})_r$ for $r \in \widetilde{\bR}_{\geq 0} = \bR_{\geq 0} \cup \{r+ \colon r \in \bR_{\geq 0}\}$ \cite[\S2]{MoyP_94}. By \cite[8.6 Cor., \S9.1]{Yu_02}, there exists a unique smooth affine $\caO_k$-model $\caG_{\bf x}^r$ of $\bfG$ satisfying $\caG_{\bf x}^r(\caO_{\breve k}) = \caG_{\bf x}(\caO_{\breve k})_r$. It is obtained from $\caG_{\bf x}$ by a series of dilatations along the unit section.

For the rest of this article we fix an integer $r \geq 1$. We consider the fpqc-quotient  
\begin{equation}\label{eq:MoyPrasad_scheme}
\bG = \bG_r = L^+\caG_{\bf x}/L^+\caG_{\bf x}^{(r-1)+}
\end{equation}
of sheaves on $\Perf_{\bF_q}$. It is representable by a (perfect) affine $\bF_q$-group scheme, perfectly of finite type over $\bF_q$ \cite[Prop. 4.2(ii)]{CI_ADLV}. We denote this group scheme, as well as its base change to $\obF$, again by $\bG$. The $\obF$-group $\bG$ admits a geometric Frobenius automorphism $F \colon \bG \rar \bG$ attached to its $\bF_q$-rational structure. We have
\[
\bG(\obF) = \caG_{\bf x}(\caO_{\breve k})/\caG_{\bf x}(\caO_{\breve k})_{(r-1)+} \quad \text{ and } \quad \bG(\bF_q) = \bG^F = 
\caG_{\bf x}(\caO_k)/\caG_{\bf x}^{(r-1)+}(\caO_k)
\]
(by taking Galois cohomology and using that $\caG_{\bf x}^{r'}$ is pro-unipotent for $r' > 0$, see \cite[\S2.6]{MoyP_94}).
For more details on this setup and for a more explicit description of $\bG$ in terms of root subgroups, we refer to \cite[\S2.4,2.5]{CI_MPDL} (such an explicit description will not be used below).

\begin{rem}\label{rem:reductive_quotient_r1}
Instead of \eqref{eq:MoyPrasad_scheme} we could work with the seemingly more natural object $L^+_r\caG_{\bf x}$ ($r$-truncated positive loops of $\caG_{\bf x}$). However, the advantage of the normalization in \eqref{eq:MoyPrasad_scheme} is that $\bG_1$ is canonically isomorphic to the reductive quotient of the special fiber $\caG_{\bf x} \otimes_{\caO_k} \bF_q$ (cf. \cite[\S3.2]{MoyP_94}), whereas $L^+_1\caG_{\bf x}$ identifies with the special fiber of $\caG_{\bf x}$, which is less useful. On the other side, if ${\bf x}$ is hyperspecial (as will be the case in our main result Theorem \ref{thm:higher_mackey}), then $L^+_r\caG_{\bf x} = \bG_r$.
\end{rem}

\subsection{Subschemes of $\bG$}\label{sec:subschemes_of_bG}
Let $\bfH \subseteq \bfG$ be a smooth closed $\breve k$-subgroup. The schematic closure $\caH \subseteq \caG_{\bf x}$ of $\bfH$ is a flat closed $\caO_{\breve k}$-subgroup scheme of $\caG_{\bf x}$ by \cite[1.2.6,1.2.7]{BruhatT_72}. Applying $L^+_r$ gives a closed immersion $L^+_r\caH_{\bf x} \subseteq L^+_r\caG_{\bf x}$ by \cite[Cor. 2 on p. 639]{Greenberg_61}. We define the closed $\obF$-subgroup $\bH \subseteq \bG$ as the image of $L^+_r\caH$ under $L^+_r\caG_{\bf x} \tar \bG$. 
If $\bfH$ is already defined over $k$, then $\caH$ is defined over $\caO_k$, and hence $\bH$ is defined over $\bF_q$. In this case, we usually will write $H := \bH(\bF_q)$.

Furthermore, for each $0<r'\leq r$, we have a natural homomorphism $\bH = \bH_r \rar \bH_{r'}$, and we denote its kernel by $\bH_r^{r'}$ (resp. simply $\bH^{r'}$).

In particular, this procedure applies to any $\bfT \in {\rm Tori}_{\breve k/k}(\bfG)$, any root subgroup $\bfU_\alpha$ (with $\alpha\in \Phi(\bfT,\bfG)$) and the unipotent radical $\bfU$ of any $\breve k$-rational Borel subgroup containing $\bfT$. This gives the subgroups $\bT,\bU_\alpha,\bU \subseteq \bG$, etc. and we will use this notation without further reference.

\subsection{Coxeter pairs and Coxeter tori}\label{sec:Cox_tori_setup}

Suppose that $\bfG$ is unramified (that is quasi-split over $k$ and split over $\breve k$). Let $\bfT_0 \subseteq \bfB_0 \subseteq \bfG$ be a $k$-rational Borel subgroup and a $k$-rational maximal torus of $\bfG$ contained in it. Let $W_0 = N_{\bfG}(\bfT_0)(\breve k)/\bfT_0(\breve k)$ be the Weyl group of $\bfT_0$. It is a Coxeter group with the set of simple reflections $S_0$ determined by $\bfB_0$. The Frobenius of $\breve k/k$ induces an automorphism $\sigma$ of $W_0$ fixing the set of simple reflections. Changing $(\bfT_0,\bfB_0)$ amounts to replacing $(W_0,S_0,\sigma)$ by a triple canonically isomorphic to it (just as in \cite[1.1]{DeligneL_76}). In particular, whenever we have a vertex ${\bf x} \in \caB(\bfG,k)$ as in Section \ref{sec:groups_setup}, we may assume that ${\bf x} \in \caA(\bfT_0,k)$.

Any pair $(\bfT,\bfB)$ with $\bfT \in {\rm Tori}_{\breve k/k}(\bfG)$ and $\bfB$ a $\breve k$-rational Borel subgroup containing it, determines the triple $(W,S,F)$, where $W$ is the Weyl group of $\bfT$, $S$ the set of simple reflections determined by $\bfB$ and $F \colon W \rar W$ is induced by the Frobenius. There is a uniquely determined coset $g\bfT_0(\breve k) \subseteq \bfG(\breve k)$ with ${}^g(\bfT_0,\bfB_0) = (\bfT,\bfB)$ and we have $g^{-1} F(g) \in N_{\bfG}(\bfT_0)(\breve k)$ mapping to some element $w = w_{\bfT,\bfB}\in W_0$. In this case, the triples $(W,S,F)$ and $(W_0,S_0,{\rm Ad}\, w \circ \sigma)$ are canonically isomorphic and we may (and will) identify them.

\begin{Def}\label{Coxeter pair} 
\begin{itemize}
\item[(i)] Given $w \in W_0$, we say that $w\sigma$ (or by abuse of language $w$) is a \emph{twisted Coxeter element} if a (any) reduced expression of $w_{\bfT,\bfB}$ contains precisely one simple reflection from any $\sigma$-orbit on $S_0$.  If $W_0$ is irreducible, the order $h$ of $w\sigma$ is called the \emph{Coxeter number} of $(W_0,\sigma)$.  
\item[(ii)] We say that $(\bfT,\bfB)$ (resp. $(\bfT,\bfU)$, where $\bfU$ is the unipotent radical of $\bfB$) is a \emph{Coxeter pair}, if $w_{\bfT,\bfB} \sigma$ is a twisted Coxeter element.
\item[(iii)] If $(\bfT,\bfB)$ is a Coxeter pair we say that $\bfT \in {\rm Tori}_{\breve k/k}(\bfG)$ is a \emph{Coxeter torus}.
\end{itemize}
\end{Def}

Recall that a torus $\bfT \in {\rm Tori}_{\breve k/k}(\bfG)$ is called \emph{elliptic} (or \emph{$k$-minisotropic}), if one of the following equivalent conditions holds: (i) $X^\ast(\bfT)^F = X^\ast(\bfZ(\bfG)^\circ)^F$, where $\bfZ(\bfG)^\circ$ is the connected component of the center of $\bfG$; (ii) the group $\bfT(k)$ has a unique fixed point (necessarily a vertex) ${\bf x} = {\bf x}_{\bfT}$ in $\caB(\bfG,k) = \caB(\bfG,\breve k)^F$. Any Coxeter torus is elliptic. Note that the property of a torus to be Coxeter (resp. elliptic) is stable under the equivalence relation of stable conjugacy.

\begin{lm}\label{lm:Coxeter_hyperspecial} Suppose $\bfG$ is unramified.
\begin{itemize}
\item[(i)] If $\bfT$ is a Coxeter torus, then ${\bf x}_{\bfT}$ is a hyperspecial vertex. 
\item[(ii)] $\bfT \mapsto {\bf x}_{\bfT}$ induces a natural bijection between $\bfG(k)$-conjugacy classes of Coxeter tori and $\bfG(k)$-orbits on the set of hyperspecial points of $\caB(\bfG,k)$.
\item[(iii)] If $(\bfT,\bfU)$, $(\bfT',\bfU')$ are Coxeter pairs with ${\bf x}_{\bfT} = {\bf x}_{\bfT'}$, then there is some $g \in \caG_{\bf x}(\caO_{\breve k})$ with ${}^g(\bfT,\bfU) = (\bfT',\bfU')$.
\end{itemize}
\end{lm}

\begin{proof}
For (i), we may pass to the adjoint group of $\bfG$. Then $\bfG \cong \prod_i \bfG_i$, with $\bfG_i$ simple and of adjoint type, and $\caB(\bfG,\breve k) \cong \prod_i \caB(\bfG_i,\breve k)$. It thus suffices to prove the result in the case $\bfG$ is simple and of adjoint type. In this case all Coxeter tori are $\bfG(k)$-conjugate by \cite[Prop. 8.1(i)]{Reeder_08}. Moreover, by \cite[Thm. 3.4.1]{deBacker_06}, there is at least the $\bfG(k)$-conjugacy class of Coxeter tori attached to a (any) hyperspecial vertex ${\bf x}$ of $\caB(\bfG,\breve k)$ and a (any) Coxeter torus in $\caG_{\bf x} \otimes_{\caO_k} \bF_q$.
Now (ii) and (iii) follow from (i), \cite[Thm. 3.4.1]{deBacker_06} and the fact that in a finite Weyl group all twisted Coxeter elements are conjugate.
\end{proof}

If $\bfT \in {\rm Tori}_{\breve k/k}(\bfG)$ is arbitrary and ${\bf x} \in \caA(\bfT, k)$, then we have the torus $\caT \subseteq \caG_{\bf x}$, and the subgroup $\bT_r \subseteq \bG_r$ for any $r>0$ (as in Section \ref{sec:subschemes_of_bG}). This gives the two Weyl groups 
\[
W_{\bf x}(\bfT,\bfG) := W(\bT_1,\bG_1) \subseteq W(\bfT,\bfG).
\]
attached to $\bfT$ (and $\bf x$). We denote them by $W_{\bf x}$ and $W$ if $\bfT,\bfG$ are clear from the context. If $\bf x$ is a hyperspecial vertex -- which is by Lemma \ref{lm:Coxeter_hyperspecial} necessarily the case whenever $\bfT$ is Coxeter -- then the situation simplifies to $\bG_1 = L_1^+\caG_{\bf x} = \caG_{\bf x} \otimes_{\bF_q} \obF$ and $W_{\bf x} = W$. 

\subsection{A condition on $q$}\label{ssec:condition}
Identifying $X_\ast(\bG_m)$ with $\bZ$, we have the perfect pairing of $\bZ$-lattices 
\[ 
X^\ast(\bT_1)  \times X_\ast(\bT_1) \rar \bZ, \quad \alpha ,\nu \mapsto \langle \alpha,\nu \rangle
\]
such that $\langle F\alpha, \nu \rangle = \langle \alpha,F \nu\rangle$ for all $\alpha,\nu$. This pairing also induces the analogous pairing for $\bT_1^{\ad}$ (where $\bfT^{\ad}$ is the image of $\bfT$ in the adjoint quotient of $\bfG$) and for the $\bQ$-vector spaces obtained by extension of scalars.

\smallskip

Recall that the choice of $\bfU$ is equivalent to the choice of a set of simple roots $\Delta \subseteq \Phi(\bfT,\bfG)$, and it endows $W$ with a structure of a Coxeter group. The simple roots $\Delta$ form a basis of $X_\ast(\bT_1^{\ad})_\bQ$. We will denote by $\{\alpha^\ast \colon \alpha \in \Delta\} \subseteq X_\ast(\bT_1^{\ad})_\bQ$ the set of fundamental coweights, defined as the basis of $X_\ast(\bT_1^{\ad})_\bQ$ dual to $\Delta$. We will prove the orthogonality relations of Coxeter-type Deligne--Lusztig characters under the following restriction on $q$:
\begin{equation}\label{cond:condition_on_rootsystem} 
q > M := \max\limits_{\alpha \in \Delta} \, \langle \alpha_0, \alpha^\ast \rangle 
\end{equation}
Note that this condition depends only on the group $\bfG_{\breve k}$ and on no other choice (like that of $\Delta$). For the irreducible types we can explicitly compute the constant $M$ from Condition \eqref{cond:condition_on_rootsystem}: type $A_n$: $M = 1$; types $B_n,C_n$: $M=2$; types $G_2,E_6$: $M = 3$; types $F_4, E_7$: $M = 4$; type $E_8$: $M = 6$. In general, the constant $M$ for $\bfG_{\breve k}$ is the maximum of the values of $M$ over all connected components of the Dynkin diagram of $\bfG_{\breve k}$. In particular \eqref{cond:condition_on_rootsystem} holds whenever $q > 5$.

\section{Deep level Deligne--Lusztig induction}\label{sec:DL_deep_induction}

We work in the setup of Section \ref{sec:groups_setup}. In particular, the reductive $\breve k$-split group $\bfG/k$, the point ${\bf x} \in \caB(\bfG,k)$, and the integer $r\geq 1$ are fixed. We omit ${\bf x}$ and $r$ from notation, and write $\caG$ for the $\caO_k$-group $\caG_{\bf x}$, and $\bG$, $\bT$, etc. for $\bG_r$, $\bT_r$, etc.

\subsection{The schemes $S_{\bfT,\bfU}$}\label{sec:DL_schemes_deep}
Let $\bfT \in {\rm Tori}_{\breve k/k}(\bfG)$, such that ${\bf x} \in \caA(\bfT, k)$. Let $\bfB = \bfT \bfU$ be a Borel subgroup, defined over $\breve k$, containing $\bfT$, and with unipotent radical $\bfU$. As in Section \ref{sec:subschemes_of_bG}, we have the corresponding closed subgroups $\bU \subseteq \bG$, defined over $\obF$. Following \cite{Lusztig_04, CI_MPDL}, consider the $\obF$-scheme
\[
\xymatrix{
S_{{\bf x}, \bfT,\bfU,r} \ar[r] \ar@{_(->}[d] & F\bU \ar@{_(->}[d] \\ \bG \ar[r]^{{\rm L}_\bG} & \bG}
\]
where $L_\bG \colon \bG \rar \bG$, $ g \mapsto g^{-1}F(g)$ is the Lang map. We usually write $S_{\bfT,\bfU}$ for $S_{{\bf x},\bfT,\bfU,r}$, as ${\bf x}$, $r$ remain constant throughout the article. The finite group $G \times T = \bG(\bF_q) \times \bT(\bF_q)$ acts on $S_{\bfT,\bfU}$ by $(g,t) \colon x \mapsto gxt$. For a character $\theta \colon T \rar \cool^\times$, we obtain the virtual $G$-representation
\[
R_{\bfT,\bfU}(\theta) = R_{{\bf x},\bfT,\bfU,r}(\theta) := \sum_{i \in \bZ} (-1)^i H_c^i(S_{\bfT,\bfU},\cool)_\theta,
\]
where the subscript $\theta$ indicates that we take the $\theta$-isotypic component. By inflation, we may regard $R_{\bfT,\bfU}(\theta)$ as a virtual smooth $\caG(\caO_k)$-representation.

\begin{rem}\label{rem:relation_to_Xwb}
The varieties $S_{\bfT,\bfU}$ are closely related to classical Deligne--Lusztig varieties. Indeed, the group $\bU \cap F\bU$ acts by right multiplication on $S_{\bfT,\bfU}$ and we may form the quotient $X_{\bfT,\bfU} = S_{\bfT,\bfU}/\bU \cap F\bU$. If $r = 1$, then $X_{\bfT,\bfU}$ is equal to the classical Deligne--Lusztig variety $\widetilde X_{\bT_1 \subseteq \bB_1}$ attached to the reductive $\bF_q$-group $\bG_1$, cf. \cite[1.17(ii), 1.19]{DeligneL_76} and Remark \ref{rem:reductive_quotient_r1}.
\end{rem}

\begin{rem}\label{rem:rel_to_Xwb_2}
In the light of Remark \ref{rem:relation_to_Xwb}, $X_{\bfT,\bfU}$ are deep level analogs of classical Deligne--Lusztig varieties. Moreover, the fibers of the morphism $S_{\bfT,\bfU} \rar X_{\bfT,\bfU}$ are isomorphic to the perfection of a fixed finite-dimensional affine space over $\bF_q$. It follows that  $H_c^\ast(S_{\bfT,\bfU}) = H_c^\ast(X_{\bfT,\bfU})$. In turn, $X_{\bfT,\bfU}$ is the $r$-truncated integral version of the $p$-adic Deligne--Lusztig spaces $X_w(b)$ (or rather their coverings $\dot X_{\dot w}(b)$) defined in \cite{Ivanov_DL_indrep}. Cf. Section \ref{sec:DL_schemes_xw} below.
\end{rem}

\subsection{Main result}\label{sec:orthogonality_relations}

Let $(\bfT,\bfU)$, $(\bfT',\bfU')$ be two pairs where $\bfT,\bfT' \in {\rm Tori}_{\breve k/k}(\bfG)$ satisfy ${\bf x} \in \caA(\bfT,k) \cap \caA(\bfT',k)$, and $\bfU$ (resp. $\bfU'$) is the unipotent radical of a $\breve k$-rational Borel subgroup of $\bfG$ containing $\bfT$ (resp. $\bfT'$). We have the groups $\caT,\bT,T,\caU,\bU$ attached to $\bfT,\bfU$ by Section \ref{sec:subschemes_of_bG}, and similarly for $\bfT',\bfU'$. 

Using Remark \ref{rem:relation_to_Xwb}, the classical orthogonality relations for Deligne--Lusztig characters \cite[Thm. 6.8]{DeligneL_76} can be expressed as follows: for $r=1$ and any characters $\theta \colon T \rar \overline \bQ_\ell^\times$, $\theta' \colon T' \rar \overline\bQ_\ell^\times$, we have 
\begin{equation}\label{eq:orthogonality_classical}
\langle R_{\bfT,\bfU}(\theta), R_{\bfT',\bfU'}(\theta') \rangle_G = \# \left\{w \in W(\bT_1,\bT_1')^F \colon \theta' = {}^w\theta \right\},
\end{equation}
where 
\begin{equation}\label{eq:transporter_TTprime}
W_{\bf x}(\bfT,\bfT') = W(\bT_1, \bT'_1) = \bT_1 \backslash \{g \in \bG_1 \colon {}^g \bT_1' = \bT_1 \} 
\end{equation}
is the transporter principal homogeneous space under $W_{\bf x}(\bfT,\bfG)$. We may ask for a generalization of this to deeper levels.

\begin{quest}\label{quest:general_orthogonality} Does \eqref{eq:orthogonality_classical} hold in general, that is for arbitrary $\bfG$, ${\bf x}$, $r$, $\bfT$, $\bfT'$, $\bfU$, $\bfU'$, $\theta$, $\theta'$? 
\end{quest}

\begin{rem}\label{rem:known_cases} The answer to Question \ref{quest:general_orthogonality} is affirmative in the following cases:
\begin{itemize}
\item[(i)] If $r=1$ by \cite[Thm. 6.8]{DeligneL_76}.
\item[(ii)] If $r \geq 2$, and $\theta$ or $\theta'$ is \emph{regular} in the sense of \cite{Lusztig_04} (roughly, ``regular'' = ``highly non-trivial on $\ker(T_r \rar T_{r-1})$'') by \cite{Lusztig_04} if $\caG$  reductive and $\charac k > 0$, resp. \cite{Stasinski_09} if $\caG$ reductive and $\charac k = 0$, resp. \cite{CI_MPDL} in general. 
\item[(iii)] If $\bfG =$ inner form of $\bfGL_n$, and $(\bfT,\bfU)$, $(\bfT',\bfU')$ are Coxeter pairs, by \cite[Thm. 3.1]{CI_loopGLn}.
\end{itemize}
\end{rem}

One might conjecture an affirmative answer to Question \ref{quest:general_orthogonality} in general, but there is not enough evidence beyond the known special cases. In this article we concentrate on the Coxeter case, and prove the following generalization of Remark \ref{rem:known_cases}(iii).

\begin{thm}\label{thm:higher_mackey}
Suppose $\bfG$ is unramified, and $(\bfT,\bfU)$, $(\bfT',\bfU')$ are Coxeter pairs with $\bf x = {\bf x}_{\bfT} = {\bf x}_{\bfT'}$. Suppose that condition \eqref{cond:condition_on_rootsystem} holds for $q$ and the root system of $\bfG$. Then for all $r \geq 1$ and all $\theta \colon T \rar \overline \bQ_\ell^\times$, $\theta' \colon T' \rar \overline \bQ_\ell^\times$, we have
\begin{equation}\label{eq:orthogonality_relation_in_theorem}
\langle R_{\bfT,\bfU}(\theta), R_{\bfT',\bfU'}(\theta') \rangle_G = \# \left\{w \in W_{\bf x}(\bfT,\bfT')^F \colon \theta' = {}^w\theta \right\},
\end{equation}
where $W_{\bf x}(\bfT,\bfT')$ is as in \eqref{eq:transporter_TTprime}.
\end{thm}

We will show Theorem \ref{thm:higher_mackey} when $(\bfT,\bfU) =(\bfT',\bfU')$ is a given Coxeter pair and $W$ is irreducible. The various reductions needed to deduce the theorem from this particular case are studied in the next section.

\section{Reductions}

The purpose of this section is to show that it is enough to prove Theorem \ref{thm:higher_mackey} when $(\bfT,\bfU) =(\bfT',\bfU')$ is a given Coxeter pair and $W$ is irreducible. There is a small price to pay, and one will actually need to show a stronger statement, namely Theorem~\ref{thm:sigmav}, which behaves well with respect to our reductions. 

\subsection{Changing Coxeter pairs}\label{sec:DL_schemes_xw}
Suppose $\bfG$ is unramified and $\bf x$ is hyperspecial. Then $\caG$ is a reductive group over $\caO_k$ and we have $\bG = L^+_r\caG$ (cf. Remark \ref{rem:reductive_quotient_r1}). Let $\bfT_0 \subseteq \bfB_0 \subseteq \bfG$ be as in Section \ref{sec:Cox_tori_setup}, such that ${\bf x} \in \caA(\bfT_0,k)$ and $W_0 = N_{\bfG}(\bfT_0)(\breve k)/\bfT_0(\breve k)$. Then $\caT_0 \subseteq \caB_0 \subseteq \caG$ are a maximal torus and a Borel subgroup containing it and defined over $\caO_k$. Let $\bfU_0$ (resp. $\caU_0$) be the unipotent radical of $\bfB$ (resp. $\caB_0$).

The $\caO_k$-group $\caG$ is quasi-split, $\caB_0 \subseteq \caG$ is a rational Borel subgroup, and the quotient $\caG/\caB_0$ is projective over $\caO_k$ \cite[Thm. 2.3.6]{Conrad_14}. Then $\caG$ admits a Bruhat decomposition in the following sense: letting $\caG$ act diagonally on $(\caG/\caB_0)^2$, there are $\caG$-stable (reduced) subschemes $\caO(w) \subseteq (\caG/\caB_0)^2$ for each $w\in W_0$, flat over $\caO_k$, such that for any geometric point $x \in \Spec \caO_k$, the fiber $\caO(w)_x$ is the $\caG_x$-orbit in $(\caG/\caB_0)^2_x$ like in the usual Bruhat decomposition. We have the following integral analog of \cite[Def. 8.3]{Ivanov_DL_indrep}.

\begin{Def}\label{def:Xc1}
Let $w \in W_0$ and $\dot w\in N_{\caG}(\caT_0)(\caO_{\breve k})$. Define the \emph{integral $p$-adic Deligne--Lusztig space} $X_w^{\caG}(1)$, and $\dot X_{\dot w}^{\caG}(1)$ by Cartesian diagrams of functors on $\Perf_{\obF}$

\centerline{\begin{tabular}{cc}
\begin{minipage}{2in}
\begin{displaymath}
\leftline{
\xymatrix{
X_w^{\caG}(1) \ar[r] \ar[d] & L^+\caO(w) \ar[d]\\
L^+(\caG/\caB_0) \ar[r]^-{(\id,F)} & L^+(\caG/\caB_0) \times L^+(\caG/\caB_0) 
}
}
\end{displaymath}
\end{minipage}
& \qquad\qquad  and \qquad\qquad
\begin{minipage}{2in}
\begin{displaymath}
\leftline{
\xymatrix{
\dot X_{\dot w}^{\caG}(1) \ar[r] \ar[d] & L^+\dot \caO(\dot w) \ar[d]\\
L^+(\caG/\caU_0) \ar[r]^-{(\id,F)} & L^+(\caG/\caU_0) \times L^+(\caG/\caU_0) 
}
}
\end{displaymath}
\end{minipage}
\end{tabular}
}
\noindent Similarly, replacing $L^+$ by $L^+_r$ everywhere, define their $r$-truncations $X_w^{\caG,r}(1)$, $\dot X_{\dot w}^{\caG,r}(1)$.
\end{Def}

The functors $X_w^{\caG}(1)$, $\dot X_{\dot w}^{\caG,r}(1)$ are representable by (perfect) $\obF$-schemes, the latter are of perfectly finite presentation. If $\dot w$ maps to $w$, then there is a natural map $\dot X_{\dot w}^{\caG,r}(1) \rar X_w^{\caG,r}(1)$; $G \times T$ acts on $\dot X_{\dot w}^{\caG,r}(1)$, $G$ acts on $X_w^{\caG,r}(1)$, and the above map is $G$-equivariant finite \'etale $T$-torsor. Recall the definition of the space $X_{\bfT,\bfU}$ from Remark \ref{rem:relation_to_Xwb}.

\begin{lm}\label{lm:XTU_Xw1}
Suppose $\bfT \in {\rm Tori}_{\breve k/k}(\bfG)$ such that ${\bf x} \in \caA(\bfT,k)$. Then $X_{\bfT,\bfU} \cong \dot X_{\dot w}^{\caG,r}(1)$ (equivariant for the $G \times T$-actions), where we identify $W(\bfT,\bfG)$ with $W_0$, $w \in W_0$ is the element satisfying $F\bfU = {}^w\bfU$, and $\dot w \in N_{\caG}(\caT_0)(\caO_{\breve k})$ is an arbitrary lift of $w$. 
\end{lm}
\begin{proof}
This has the same proof as \cite[1.19]{DeligneL_76}. There are no subtleties due to the loop functor, cf. the similar results of \cite[Prop. 12.1 and Lem. 12.3]{Ivanov_DL_indrep}.
\end{proof}

To $\dot X_{\dot w}^{\caG,r}(1)$ we may apply the technique of Frobenius-cyclic shift. Let $\ell$ denote the length function on the Coxeter group $(W_0,S_0)$. 

\begin{lm}\label{lm:Fcyclic_shift}
Suppose $w = w_1w_2$, $w' = w_2F(w_1) \in W_0$, such that $\ell(w) = \ell(w_1) + \ell(w_2) = \ell(w')$. Then there is a $G$-equivariant isomorphism $X_w^{\caG,r}(1) \cong X_{w'}^{\caG,r}(1)$. If $\dot w$, $\dot w'$, $\dot w_1$, $\dot w_2 \in \caG(\caO_{\breve k})$ are lifts of $w,w',w_1,w_2$, satisfying $\dot w = \dot w_1 \dot w_2$, $\dot w' = \dot w_2 F(\dot w_1)$, then there is a $G \times T$-equivariant isomorphism $\dot X_{\dot w}^{\caG,r}(1) \cong \dot X_{\dot w'}^{\caG,r}(1)$.
\end{lm}

\begin{proof}
The same proof as in \cite[1.6]{DeligneL_76} applies. Again, the use of the (positive, truncated) loop functor causes no problems, cf. \cite[Lem. 8.16]{Ivanov_DL_indrep}.
\end{proof}

As a corollary we deduce:

\begin{cor}\label{cor:replaceTU_TU}
Suppose $\bfG$ is unramified, and $(\bfT,\bfU)$, $(\bfT',\bfU')$ are Coxeter pairs with ${\bf x} = {\bf x}_\bfT = {\bf x}_{\bfT'}$ (in particular, ${\bf x}$ hyperspecial). Then $X_{\bfT,\bfU} \cong X_{\bfT',\bfU'}$ ($G\times T \cong G \times T'$-equivariantly). In particular, $H^\ast_c(S_{\bfT,\bfU}) \cong H^\ast_c(S_{\bfT',\bfU'})$. To show Theorem \ref{thm:higher_mackey}, it suffices to do so under the additional assumption $(\bfT',\bfU') = (\bfT,\bfU)$ is a fixed Coxeter pair.
\end{cor}
\begin{proof}
We prove the first statement. By Lemma \ref{lm:XTU_Xw1} it suffices to show that whenever $w,w'\in W_0$ are two twisted Coxeter elements, $\dot X_{\dot w}^{\caG,r}(1) \cong \dot X_{\dot w'}^{\caG,r}(1)$. First, when $\dot w_1,\dot w_2 \in N_{\caG}(\caT_0)(\caO_{\breve k})$ are two lifts of $w$, then $\dot X_{\dot w_1}^{\caG,r}(1) \cong \dot X_{\dot w_2}^{\caG,r}(1)$ equivariantly (same argument as on \cite[p. 111]{DeligneL_76}, along with an application of Lang's theorem to the connected $\obF$-group  $L^+_r\caT_0$ with Frobenius ${\rm Ad}(w) \circ F$). Using this, the first statement of the corollary follows from Lemma \ref{lm:Fcyclic_shift} along with the fact that all twisted Coxeter elements are conjugate by a sequence of cyclic shifts in $W_0$ (cf. the corresponding discussion in \cite[\S8.4]{Ivanov_DL_indrep}). 

The second claim follows from the first and Remark \ref{rem:rel_to_Xwb_2}, and the third claim follows from the second.
\end{proof}

\subsection{First step towards the proof of Theorem \ref{thm:higher_mackey}}\label{sec:general_strategy}

In the proof of Theorem \ref{thm:higher_mackey} we follow the general strategy of \cite[\S6]{DeligneL_76} and \cite{Lusztig_04}. Let the setup be as in the beginning of Section \ref{sec:orthogonality_relations}. Attached to $(\bfT,\bfU)$, $(\bfT',\bfU')$ we may consider the $\obF$-scheme
\begin{align}
\label{eq:definition_of_Sigma}\Sigma := {}^{\bU,\bU'}\Sigma &:= G \backslash (S_{\bfT,\bfU} \times S_{\bfT',\bfU'}) \\ 
\nonumber &\mbox{}\,\,\cong \{(x,x',y) \in F\bU \times F\bU' \times \bG \colon xF(y) = yx' \}.
\end{align}
We will write ${}^{\bU,\bU'}\Sigma$, whenever the choice of $\bU,\bU'$ is relevant, and simply $\Sigma$ whenever it is clear from context). In \eqref{eq:definition_of_Sigma} the group $G$ acts diagonally on $S_{\bfT,\bfU} \times S_{\bfT',\bfU'}$, and the second isomorphism is given by $(g,g') \mapsto (x,x',y)$ with $x = g^{-1}F(g)$, $x' = g^{\prime -1}F(g')$, $y = g^{-1}g'$, just as in \cite[6.6]{DeligneL_76}. Now $T\times T'$ acts on $\Sigma$ by $(t,t') \colon (x,x',y) \mapsto (txt^{-1},t'x't^{\prime -1}, tyt^{\prime -1})$ and an application of the K\"unneth formula shows that 
\[
\left\langle R_{\bfT,\bfU}(\theta), R_{\bfT',\bfU'}(\theta')\right\rangle_G = \dim_{\cool} H_c^\ast(\Sigma,\cool)_{\theta \otimes \theta'}.
\]
Let $\pr \colon \bG = \bG_r \rar \bG_1$ denote the natural projection. We have the locally-closed decomposition $\bG_1 = \coprod_{v \in W(\bT_1,\bT'_1)} \bU_1 \dot v \bT_1' \bU_1'$. This induces a $T\times T'$-stable locally closed decomposition $\Sigma = \coprod_{v \in W(\bT_1,\bT'_1)} \Sigma_v$, with 
\begin{equation}\label{eq:Sigma_v_def}
\Sigma_v = \{(x,x',y) \in \Sigma \colon y \in \pr^{-1}(\bU_1 \dot v \bT_1' \bU_1') \}, 
\end{equation}
where $\dot v$ is an arbitrary lift of $v$ to $\bG(\obF)$ fixed once and for all.
To prove formula \eqref{eq:orthogonality_relation_in_theorem} for the given $\bfT,\bfT',\bfU,\bfU',\theta,\theta'$ it suffices to show
\begin{equation}\label{eq:sigma_v_coh}
\dim_{\cool} H_c^\ast(\Sigma_v,\cool)_{\theta\otimes\theta'} = \begin{cases} 1 & \text{if $F(v) = v$ and $\theta' = {}^v\theta$,} \\ 0 & \text{otherwise.} \end{cases} 
\end{equation}

We shall show that a stronger statement holds for a specific choice of a Coxeter pair. Let $\bZ = Z(\bG)$ be the center of $\bG$ and $Z := \bZ^F$ be its rational points. The group $Z$ embeds diagonally in $T \times T'$ and its action on $\Sigma$ (hence on its cohomology) is trivial. The action of $T \times^Z T'$ on $\Sigma$ extends to an action of $(\bT \times^\bZ \bT')^F$ and the cohomology of the cell $\Sigma_v$ for that action is given by the following theorem.

\begin{thm}\label{thm:sigmav}
Suppose $\bfG$ is unramified, and that condition \eqref{cond:condition_on_rootsystem} holds for $q$ and the root system of $\bfG$. 
Then there exists a Coxeter pair $(\bT,\bU)$ such that for all $v \in W$
\begin{equation}\label{eq:vTF_and_Sigma_v}
H_c^\ast(\Sigma_v^{\bU,\bU}) = \begin{cases} H_c^0((\dot v \bT)^F, \overline \bQ_\ell) & \text{if $v \in W^F$,} \\ 0 &\text{otherwise,} \end{cases}
\end{equation}
as virtual $\cool (\bT \times^\bZ \bT')^F$-modules.
\end{thm}

Equation \eqref{eq:sigma_v_coh} follows easily from this theorem. Indeed, if $\theta_{|Z} \neq \theta_{|Z}'$ then $H_c^\ast(\Sigma_v)_{\theta \otimes \theta'} = 0$ since $Z$ acts trivially on $\Sigma_v$. On the other hand, since $ T \times^Z T' \subset (\bT \times^\bZ \bT')^F$, Theorem~\ref{thm:sigmav} implies that the cohomology of $\Sigma_v$ as a virtual $T\times^Z T$-module is the same as the cohomology of $(\dot v \bT)^F$ for which the analog of  \eqref{eq:sigma_v_coh} clearly holds.

\smallskip

The proof of Theorem~\ref{thm:sigmav} in the case where $W$ is irreducible will be given in Section~\ref{sec:proof_for_aasG}. The reduction to that case is the purpose of the remainder of this section.

\subsection{Reduction to the almost simple case}\label{ssec:reduction-to-simple}
Let $\bfG$ be an arbitrary unramified reductive group over $k$. Let $\pi \colon \widetilde \bfG \rar \bfG$ be the simply connected covering of the derived group of $\bfG$. Let $\bfZ$ denote the center of $\bfG$. Adjoint buildings of $\bfG$ and $\widetilde \bfG$ agree, and we have the parahoric $\caO_k$-model $\widetilde \caG$ of $\widetilde \bfG$, corresponding to the same point as $\caG$. Moreover, $\pi$ extends uniquely to a map $\pi \colon \widetilde \caG \rar \caG$ \cite[1.7.6]{BruhatT_84}, which in turn induces the map $\pi \colon \widetilde \bG \rar \bG$. Put $\widetilde \bfT = \pi^{-1}(\bfT)$, $\widetilde \caT = \pi^{-1}(\caT)$, $\widetilde \bT = \pi^{-1}(\bT)$, and similarly for $\bfU$, $\bfZ$, etc.

\begin{rem}\label{rem:case_of_r1}
If $r=1$, then $\widetilde\bG \rar \bG$ is the simply connected cover of the derived group of $\bG$, and the situation is precisely as in \cite[1.21-1.27]{DeligneL_76}). 
\end{rem}

The map $\pi$ induces maps on rational points $\widetilde G = \widetilde \bG(\bF_q) \rar \bG(\bF_q) = G$, and similarly $\widetilde T \rar T$. In particular, any character $\chi$ of $T$ pulls back to a character  $\widetilde \chi$ of $\widetilde T$. Now the general case of Theorem \ref{thm:higher_mackey} follows from the next proposition.

\begin{prop}\label{prop:reduction_to_sssc}
If Theorem \ref{thm:sigmav} holds for $\widetilde \bfG$, $\widetilde \bfT$ then it holds for $\bfG$, $\bfT$. 
\end{prop}

\begin{proof}
Let $\bS = \bT \times \bT$ and $S = T \times T$ (resp. $\widetilde \bS = \pi^{-1}(\bS)$ and $\widetilde S = \widetilde \bS(\bF_q)$). We have the space $X$ and $\widetilde X$ carrying actions of $G \times T$ and $\widetilde G \times \widetilde T$ respectively. Moreover, we also have the quotients $\Sigma = (X\times X)/G$ and $\widetilde \Sigma = (\widetilde X \times \widetilde X)/\widetilde G$ acted on by $S$ and $\widetilde S$ respectively. The $G\times T$-action on $X$ factors through the action of the quotient $G \times^Z T = G\times T/\{(z,z^{-1}) \colon z \in Z\}$, which in turn extends to an action of the bigger group $(\bG \times^\bZ \bT)^F$ given by the same formula. Similarly, the $S$-action on $\Sigma$ factors through an action of $S/Z$ ($Z$ embedded diagonally), which extends to an action of $(\bS/\bZ)^F$ given by same formula. These two extensions of actions also hold when we put a $\widetilde{(\cdot)}$ over each of the objects.

Recall the notion of the \emph{induced space} from \cite[1.24]{DeligneL_76}: if $\alpha \colon A \rar B$ is a homomorphism of finite groups, and $Y$ a space on which $A$ acts, then the induced space ${\rm Ind}_A^B Y = {\rm Ind}_\alpha Y$ is the (unique up to unique isomorphism) $B$-space $I$, provided with an $A$-equivariant map $Y \rar I$, which satisfies $\Hom_B(I,V) = \Hom_A(Y,V)$ for any $B$-space $V$. 

\begin{lm}\label{lm:induced_Sigma_space}
Let $\gamma \colon (\widetilde \bS /\widetilde \bZ)^F \rar (\bS /\bZ)^F$ be the natural map induced by $\pi$. Then $\Sigma = \Ind_\gamma \widetilde \Sigma$.
\end{lm}
\begin{proof}
We have the natural map $\alpha \colon (\widetilde \bS \times^{\widetilde\bZ} \widetilde \bG)^F \rar (\bS \times^\bZ \bG)^F$. Kernel and cokernel of $\alpha$ are canonically isomorphic to the kernel and cokernel of $\beta \colon \widetilde \bS^F \rar \bS^F$ (same argument as \cite[1.26]{DeligneL_76} with $\bS$ instead of $\bT$). One checks that $X = \Ind_{\widetilde \bS^F}^{\bS^F} \widetilde X$. Thus, similar as in \cite[1.25]{DeligneL_76}, 
\begin{equation}\label{eq:induction_XX_XX}
X \times X = \Ind\nolimits_\beta \widetilde X \times \widetilde X = \Ind\nolimits_\alpha \widetilde X \times \widetilde X.
\end{equation}
Now, we have the commutative diagram with exact rows:
\begin{equation}\label{eq:SCG_relation}
\xymatrix{
1 \ar[r] & \widetilde \bG^F \ar[r] \ar[d] & (\widetilde \bS \times^{\widetilde \bZ} \widetilde \bG)^F \ar[r] \ar[d]^{\alpha} & (\widetilde \bS/\widetilde \bZ)^F \ar[r] \ar[d]^\gamma & 1 \\
1 \ar[r] & \bG^F \ar[r] & (\bS \times^{\bZ} \bG)^F \ar[r] & (\bS/\bZ)^F \ar[r] & 1 
}
\end{equation}
which is obtained from the same diagram for the algebraic groups (with all $F$'s removed) by taking Galois cohomology and using Lang's theorem and connectedness of $\bG$, $\widetilde\bG$. Now the lemma formally follows from \eqref{eq:SCG_relation} and \eqref{eq:induction_XX_XX}, using that $\Sigma = (X\times X)/\bG^F$ and $\widetilde \Sigma =( \widetilde X \times \widetilde X) / \widetilde \bG^F$.
\end{proof}
 
On the other hand, if $v \in W$ then we also have  $(\dot v \bT)^F = \Ind_\gamma  (\dot v \widetilde\bT)^F$ as $(\bS/\bZ)^F$-varieties.   Therefore Proposition~\ref{prop:reduction_to_sssc} follows from Lemma~\ref{lm:induced_Sigma_space}.\end{proof}

Assume now that $\bfG$  is semisimple and simply connected. In this case, there is some $s\geq 1$ such that $\bfG \cong \prod_{i=1}^s \bfG_i$, where each $\bfG_i$ is an almost simple and simply connected unramified reductive $k$-group. We have then similar product decompositions for the Bruhat--Tits buildings, the parahoric models $\caG \cong \prod_i \caG_i$, their Moy--Prasad filtrations, the maximal tori, their Weyl groups, the unipotent radicals of the Borels, etc. Upon applying the functor $\caG \mapsto L^+\caG/L^+\caG^{(r-1)+}$, this induces an isomorphism $X_{\bfT,\bfU} \cong \prod_i X_{\bfT_i,\bfU_i}$, and finally an isomorphism $\Sigma \cong \prod_i \Sigma_i$ (with obvious notation), equivariant for the action of $T \times T = \prod_i (T_i \times T_i)$. Applying the K\"unneth formula shows that Theorem~\ref{thm:sigmav} holds for $\bfG,\bfT$ whenever it holds for all $\bfG_i$, $\bfT_i$. 

\smallskip

Finally if $\bfG$ is almost simple and simply connected, then there is some $m \geq 1$, and an absolutely almost simple group $\widehat\bfG$ over $k_m$, the degree $m$ subextension of $\breve k/k$, such that $\bfG \cong \Res_{k_m/k} \widehat\bfG$ is the restriction of scalars of $\widehat\bfG$. We thus may assume that $\bfG = \Res_{k_m/k} \widehat\bfG$.
The Bruhat--Tits buildings $\cB(\bfG,k)$ and $\cB(\widehat\bfG,k_m)$ are canonically isomorphic. Let ${\bf x}$ be a vertex of $\cB(\bfG,k)$ with attached parahoric $\caO_k$-model $\caG$ of $\bfG$, and let ${\bf x}$ also denote the corresponding vertex of $\cB(\widehat\bfG,k_m)$, with attached parahoric $\caO_{k_m}$-model $\widehat\caG$ of $\widehat\bfG$. Then there is a canonical isomorphism $\caG = \Res_{\caO_{k_m}/\caO_k} \widehat\caG$ inducing the identity on generic fibers \cite[Prop. 4.7]{HainesR_20}. Reducing modulo $\varpi^r$ and applying \cite[Cor. 13.5]{BertapelleG_18} (with $e=1$), we deduce a canonical identification $\bG = \Res_{\bF_{q^m}/\bF_q} \widehat\bG$, where $\bG,\widehat\bG$ are attached to $\caG,\widehat\caG$ as in Section \ref{sec:subschemes_of_bG}.

We have $\bfT = \Res_{k_m/k} \widehat\bfT$ for a Coxeter torus of $\widehat\bfG / k_m$, and we may identify $W = \prod_{i=1}^m \widehat W$, where $\widehat W$ is the Weyl group of $\widehat\bfT$. Furthermore under this identification one can assume that $F$ acts by $F((w_i)_{i=1}^m) = (\widehat F(w_m),w_1, \dots,w_{m-1})$, where $\widehat F$ is the Frobenius of $\widehat W$. In particular $W^F = \{(w_1,\dots,w_1) \colon \widehat F(w_1) = w_1\} \simeq (\widehat W)^{\widehat F}$. Choose $\bfU$ such that $F(\bfU) = {}^c\bfU$, where $c = (\widehat c,1,\dots,1) \in W$ and $\widehat c \in \widehat W$ is the twisted Coxeter element of $\widehat W$ satisfying $\widehat F(\widehat \bfU) = {}^{\widehat c} \widehat\bfU$. Then $(\bfT,\bfU)$ is a Coxeter pair. Now, if we consider the decomposition $\bG_{\obF} \cong \prod_{i=1}^m \widehat\bG_{\obF}$ the equation $ xF(y) = yx' $
for $(x, x',y) \in  \bU \times \bU \times \bG$ can be written 
$$ (x_1,\dots,x_m)(\widehat F(y_m),y_1,\dots, y_m) = (y_1,\dots,y_m)(x_1',\dots, x_m'),$$
which in turn is equivalent to
$$\widehat x \widehat F(y_1) = y_1\widehat x' \quad \text{and} \quad \forall i \in \{2,\ldots,m\} \ y_i = x_i y_{i-1} (x_i')^{-1}$$
where  $\widehat x := x_1 \widehat F(x_m x_{m-1} \dots x_1)$ and $\widehat x' := x_1'F(x_m'x_{m-1}'\dots x_1')$. Therefore we can remove all the $y_i$'s for $i \geq 2$ to show that 
$$\Sigma = \{(x_i),(x'_i),y_1 \in \bU \times \bU \times \widehat\bG \colon \widehat x \widehat F(y_1) = y_1\widehat x' \}.$$
This scheme lies over the scheme $\widehat\Sigma = \{\widehat x,\widehat x',y_1 \in \widehat\bU \times \widehat\bU \times \widehat\bG \colon \widehat x \widehat F(y) = y \widehat x'\}$ attached to $\widehat\bG$, via the natural map $(x_i),(x_i'),y_1 \mapsto (\widehat x, \widehat x',y_1)$. All fibers of this map are isomorphic to the perfection of a fixed affine space of some dimension, so that $H_c^\ast(\Sigma) = H_c^\ast(\widehat\Sigma)$. 
This shows that Theorem~\ref{thm:sigmav} holds for $\bfG$ whenever it holds for $\widehat\bfG$

\smallskip

Summarising the results obtained in this section, we have that Theorem~\ref{thm:sigmav} holds whenever it holds for any absolutely almost simple group. In particular we shall, and we will, only consider the case where $W$ is irreducible in Sections \ref{sec:regularity} and \ref{sec:proof_for_aasG}.

\section{Extensions of action}\label{sec:extension_of_action}

Throughout this section we work in the general setup of Section \ref{sec:groups_setup}. We fix two \emph{arbitrary} pairs $(\bfT,\bfU)$, $(\bfT',\bfU')$ with ${\bf x} \in \caA(\bfT,k) \cap \caA(\bfT',k)$. Then we have the corresponding subgroups $\bU,\bU' \subseteq \bG$ and for $v \in W(\bT_1,\bT_1')$, the scheme $\Sigma_v = {}^{\bU,\bU'}\Sigma_v$ as in \eqref{eq:Sigma_v_def}.
Pushing further the ideas from \cite[(6.6.2)]{DeligneL_76} and \cite[\S3.3 and \S3.4]{CI_loopGLn}, we will extend the action of the finite group $T \times T'$ on $\Sigma_v$ to the actions of various bigger groups.

\subsection{Lusztig's extension}\label{sec:Lusztigs_extension}
First we have the extension of action due to Lusztig (and a minimal variation of it). The geometric points $\bG^1(\obF)$ of the group $\bG^1 = \ker(\bG \rar \bG_1)$ can be written as a product of all ``root subgroups'' $\bU_{\alpha}(\obF)$, $\alpha \in \Phi(\bfT,\bfG)$ contained in it, and these can be taken in any order \cite[(6.4.48)]{BruhatT_72}, so we have 
\begin{align*} \pr^{-1}(\bU_1 \dot v \bT_1' \bU_1') &= \bU\dot v \bG^1 \bT' \bU' \\ 
&= \bU \dot v \left[ ({}^{v^{-1}}\bU)^1 ({}^{v^{-1}}\bU^- \cap \bU^{\prime -})^1 \bU^{\prime 1} \bT^{\prime 1}\right]\bT' \bU' \\
&= \bU\dot v \bT^{\prime}\bK^1 \bU',
\end{align*}
where we put 
\[ \bK := \bK_{\bU,\bU',v} := {}^{v^{-1}}\bU^- \cap \bU^{\prime -} \]
Then $\Sigma_v$ can be rewritten as 
\[ \Sigma_v = \{(x,x',y) \in \Sigma \colon y \in \bU\dot v \bT^{\prime}\bK^1 \bU'\}. \]
Consider
\begin{equation}\label{eq:first_presentation_widehat_Sigma}
\widehat\Sigma_v = \{x,x',y',\tau,z,y'' \in F\bU \times F\bU' \times \bU \times \bT' \times \bK^1 \times \bU' \colon xF(y'\dot v\tau z y'') = y'\dot v\tau z y''x' \},
\end{equation}
with an action of $T\times T'$ given by
\[ (t,t') \colon (x,x',y',\tau,z,y'') \mapsto (txt^{-1},t'x't^{\prime -1},ty't^{-1},\dot v^{-1}t\dot v\tau t',t'zt^{\prime -1},t'y''t^{\prime -1}). 
\]
Then we have an obvious $T\times T'$-equivariant map $\widehat\Sigma_v \rar \Sigma_v$, $(x,x',y',\tau,z,y'') \mapsto (x,x',y'\dot v \tau z y'')$, which is a Zariski-locally trivial fibration with fibers isomorphic to the perfection of a fixed affine space. Then the $\ell$-adic Euler characteristic does not change, so that we have an equality of virtual $T \times T'$-modules
\begin{equation}\label{eq:coh_Sigma_Sigma_hat_same}
H_c^\ast(\Sigma_v) = H_c^\ast(\widehat\Sigma_v).
\end{equation}

Now make the change of variables $xF(y') \mapsto x$, $x'F(y'')^{-1} \mapsto x'$, so that 
\begin{equation}\label{eq:widehat_Sigma_after_change_var}
\widehat\Sigma_v \cong \{x,x',y',\tau,z,y'' \in F\bU \times F\bU' \times \bU \times \bT' \times \bK^1 \times \bU' \colon xF(\dot v \tau z) = y'\dot v \tau z y'' x'\},
\end{equation}

\begin{lm}[\cite{Lusztig_04}, 1.9]\label{lm:Lusztigs_extension_of_action} (i) This $T \times T'$-action on $\widehat\Sigma_v$ extends to an action of the closed subgroup 
\[
H_v = \{(t,t') \in \bT \times \bT' \colon \dot v^{-1} t^{-1}F(t)\dot v = t^{\prime -1}F(t') \text{ centralizes $\bK = ({}^{v^{-1}}\bU \cap \bU')^-$} \}
\]
of $\bT \times \bT'$, given by
\[
(t,t') \colon (x,x',y',\tau,z,y'') \mapsto (F(t)xF(t)^{-1}, F(t')x'F(t')^{-1}, F(t)y'F(t)^{-1},\dot v^{-1}t\dot v\tau t^{\prime -1},t'zt^{\prime -1},F(t')y''F(t')^{-1}).
\]
(ii) Similarly, the $T\times T'$-action on $\widehat\Sigma_v$ extends to an action of the closed subgroup
\[
H_v' = \{(t,t') \in \bT \times \bT' \colon F(\dot v)^{-1} tF(t)^{-1}F(\dot v) = t'F(t')^{-1} \text{ centralizes $F(\bK) = F({}^{v^{-1}}\bU \cap \bU')^-$} \}
\]
of $\bT \times \bT'$, given by
\[
(t,t') \colon (x,x',y',\tau,z,y'') \mapsto (txt^{-1}, t'x't^{\prime -1}, ty't^{-1},\dot v^{-1}t\dot v\tau t^{\prime -1},t'zt^{\prime -1},t'y''t^{\prime -1}).
\]
\end{lm}
\begin{proof}
(ii) is proven in \cite[1.9]{Lusztig_04}. The proof of part (i) is completely analogous.
\end{proof}

\subsection{Another extension of action}\label{sec:another_extension_of_action} To extend the action differently, we replace the resolution $\widehat\Sigma_v \rar \Sigma_v$ by a different one. For that purpose note that the ($\obF$-points of the) closed subgroup 
\begin{equation}\label{eq:some_group_on_rhs}
({}^{v^{-1}}\bU^- \cap \bU^{\prime -})^1 \cdot ({}^{v^{-1}}\bU^- \cap \bU')
\end{equation}
of ${}^{v^{-1}}\bU^-$ can be described as being cut out by a certain concave function on $\Phi(\bfT,\bfG)$. More precisely, it is equal to the quotient of $U_f$ (in the sense of Bruhat--Tits \cite[\S6.2]{BruhatT_72}) with $f \colon \Phi \cup \{0\} \rar \widetilde\bR$ being the function 
\[
f(\alpha) = \begin{cases} \infty & \text{if $\alpha = 0$ or $\alpha \in \Phi(\bfT, {}^{v^{-1}}\bfU)$} \\ 0 & \text{if $\alpha \in \Phi(\bfT,{}^{v^{-1}} \bfU \cap \bfU')$}\\ 1 & \text{if $\alpha \in \Phi(\bfT, {}^{v^{-1}} \bfU \cap \bfU^{\prime -})$}\\ \end{cases}
\]
by the normal subgroup $\ker(L^+\caG \rar \bG_r)(\obF)$ (which is itself of the form $U_{f'}$ for a further concave function $f'$). By \cite[6.4.48]{BruhatT_72}, the order of the roots in the product expression appearing in \eqref{eq:some_group_on_rhs} can be chosen arbitrary, thus the group \eqref{eq:some_group_on_rhs} is also equal to
\begin{align}
\label{eq:new_expression_of_intermed_product_groups} &({}^{v^{-1}}\bU^- \cap \bU^{\prime -} \cap F\bU^{\prime -})^1 \cdot ({}^{v^{-1}}\bU^- \cap \bU^{\prime} \cap F\bU^{\prime -}) \cdot ({}^{v^{-1}}\bU^- \cap \bU^{\prime -} \cap F\bU^{\prime})^1 \cdot ({}^{v^{-1}}\bU^- \cap \bU^{\prime} \cap F\bU^{\prime})  \\
\nonumber&= ({}^{v^{-1}}\bU^- \cap F\bU^{\prime -})' \cdot ({}^{v^{-1}}\bU^- \cap F\bU^{\prime})',
\end{align}
where $({}^{v^{-1}} \bU^- \cap F\bU^{\prime -})'$ denotes the closed subgroup of ${}^{v^{-1}} \bU^- \cap F\bU^{\prime -}$ determined by the appropriate concave function on roots (and similarly for $({}^{v^{-1}}\bU^- \cap F\bU^{\prime})'$). Then on $\obF$-points we have
\begin{align*}
\pr^{-1}(\bU_1\dot v \bT_1' \bU'_1) &= \pr^{-1}(\bU_1\dot v \bT_1' ({}^{v^{-1}}\bU_1^- \cap \bU'_1))  \\
&= \bU\dot v \bT' \bG^1 ({}^{v^{-1}}\bU^- \cap \bU'_1) \\
&= \bU\dot v\bT'({}^{v^{-1}} \bU^- \cap \bU^{\prime -})^1  ({}^{v^{-1}} \bU^- \cap \bU^{\prime}),
\end{align*}
using that $\bG^1(\obF)$ decomposes into the product of ``root subgroups'' $\bU_\alpha(\obF)$ contained in it, taken in any order. Using this and the expression \eqref{eq:new_expression_of_intermed_product_groups} of the group \eqref{eq:some_group_on_rhs}, we can rewrite
\[ \Sigma_v = \{x,x',y \in F\bU \times F\bU' \times \bU\dot v \bT' ({}^{v^{-1}}\bU^- \cap F\bU^{\prime -})' ({}^{v^{-1}} \bU^- \cap F\bU^{\prime})' \colon xF(y) = yx' \}. \]
Now consider 
\[
\widetilde\Sigma_v := \{x,x',y',\tau,z_1,y_1'' \in F\bU \times F\bU' \times \bU \times \bT' \times ({}^{v^{-1}}\bU^- \cap F\bU^{\prime -})' \times ({}^{v^{-1}} \bU^- \cap F\bU^{\prime})' \colon xF(y'\dot v \tau z_1y_1'') = y'\dot v \tau z_1y_1''x' \},
\]
with $T\times T'$-action given by 
\[(t,t') \colon (x,x',y',\tau,z_1,y_1'') \mapsto (txt^{-1},t'x't^{\prime -1}, ty't^{-1},\dot v^{-1}t\dot v \tau t', t'z_1t^{\prime -1}, t'y_1''t^{\prime -1})\]
Then the map $\widetilde\Sigma_v \rar \Sigma_v$, $(x,x',y',\tau,z_1,y_1'') \mapsto (x,x',y'\dot v \tau z_1 y_1'')$ is a $T\times T'$-equivariant Zariski-locally trivial fibration, with fibers isomorphic to the perfection of some fixed affine space. In particular, we again have an equality of virtual  $T \times T'$-modules
\begin{equation}\label{eq:coh_Sigma_Sigma_tilde_same}
H_c^\ast(\Sigma_v) = H_c^\ast(\widetilde\Sigma_v) 
\end{equation}

Now we make the change of variables $xF(y') \mapsto x$, $y_1'' x' \mapsto x'$, so that
\[
\widetilde\Sigma_v := \{x,x',y',\tau,z_1,y_1'' \in F\bU \times F\bU' \times \bU \times \bT' \times ({}^{v^{-1}}\bU^- \cap F\bU^{\prime -})' \times ({}^{v^{-1}} \bU^- \cap F\bU^{\prime})' \colon xF(\dot v \tau z_1y_1'') = y'\dot v \tau z_1 x' \},
\]
with the action of $T\times T'$ given by the same formula as before.
\begin{lm}\label{lm:new_simple_extensions}
\begin{itemize}
\item[(i)] The action of $T\times T'$ on $\widetilde\Sigma_v$ extends to an action of the closed subgroup
\[
H_v'' = \{(t,t') \in \bT \times \bT' \colon \dot v^{-1}F^{-1}(t)^{-1}t\dot v = t^{\prime -1}F^{-1}(t') \text{ centralizes } {}^{v^{-1}}\bU^- \cap F\bU^{\prime -}\},
\]
of $\bT\times \bT'$ given by
\[(t,t') \colon (x,x',y',\tau,z_1,y_1'') \mapsto (txt^{-1},t'x't^{\prime -1}, ty't^{-1},\dot v^{-1}t\dot v \tau t^{\prime -1}, t'z_1t^{\prime -1}, F(t')y_1''F(t^{\prime})^{-1})\]

\item[(ii)] The action of $T\times T'$ on $\widetilde\Sigma_v$ extends to an action of the closed subgroup
\[
H_v''' = \{(t,t') \in \bT \times \bT' \colon \dot v^{-1}F^{-1}(t)^{-1}t\dot v = t^{\prime -1}F^{-1}(t') \text{ centralizes } {}^{v^{-1}}F\bU^- \cap \bU^{\prime -}\},
\]
of $\bT\times \bT'$ given by
\[
(t,t') \colon (x,x',y',\tau,z_1,y_1'') \mapsto (txt^{-1},t'x't^{\prime -1}, ty't^{-1},\dot v^{-1}t\dot v \tau t^{\prime -1}, t'z_1t^{\prime -1}, F^{-1}(t')y_1''F^{-1}(t^{\prime})^{-1})
\]
\end{itemize}
\end{lm}
\begin{proof} The proof is a computation similar to Lemma \ref{lm:Lusztigs_extension_of_action}.
\end{proof}

\subsection{An isomorphism}\label{sec:isomorphism}
The extensions of actions from Sections \ref{sec:Lusztigs_extension} and \ref{sec:another_extension_of_action} suffice to prove Theorem \ref{thm:higher_mackey} in type $A_n$, as was done in \cite[Thm. 3.1]{CI_loopGLn}. The proof was however based on a particular combinatorial property of this type. For the general case, we need the following new idea. One immediately checks that 
\begin{align*}
\alpha = {}^{\bU,\bU'}\alpha \colon {}^{\bU,\bU'}\Sigma &\rar {}^{\bU,F\bU'}\Sigma \\
(x,x',y) &\mapsto (x,F(x'),yx')
\end{align*}
is an $T \times T'$-equivariant isomorphism. In general, it does not preserve the locally closed pieces $\Sigma_v$. However, we have the following lemma.

\begin{lm}\label{lm:alpha_image_stable_under_Lusztigs_extension}
For $v,w \in W(\bT_1,\bT_1')$, let $Y_{v,w}\subseteq {}^{\bU,F\bU'}\widehat\Sigma_w$ be defined by the Cartesian diagram
\[
\xymatrix{
Y_{v,w} \ar[r] \ar[d] & {}^{\bU,F\bU'}\widehat\Sigma_w \ar[d] \\
\alpha({}^{\bU,\bU'}\Sigma_v) \cap {}^{\bU,F\bU'}\Sigma_w \ar[r] & {}^{\bU,F\bU'}\Sigma_w
}
\]
where the left lower entry is the scheme-theoretic intersection inside ${}^{\bU,F\bU'}\Sigma_w$. Then $Y_{v,w}$ is stable under the action of $H_w$ on ${}^{\bU,F\bU'}\widehat\Sigma_w$ defined in Lemma \ref{lm:Lusztigs_extension_of_action}(i).
\end{lm}
\begin{proof} In terms of the presentation \eqref{eq:first_presentation_widehat_Sigma} of ${}^{\bU,F\bU'}\widehat\Sigma_w$ where we denote the coordinates by $x_1,x_1',y_1',\tau_1,z_1,y_1''$, consider the morphism 
\[
y_1'\dot w \tau_1 z_1 y_1'' F^{-1}(x_1')^{-1} \colon \widehat\Sigma_w \rar \bG.
\]
The subscheme $Y_{v,w}$ is the preimage under this morphism of $\pr^{-1}(\bU_1 \dot v \bT_1'\bU_1')$. Now we apply the change of coordinates ($x_1F(y_1) \mapsto x_1$, $x_1'F(y_1'')^{-1} \mapsto x_1'$) from \eqref{eq:first_presentation_widehat_Sigma} to \eqref{eq:widehat_Sigma_after_change_var}. Then the expression $F^{-1}(x_1')^{-1}$ in the old coordinates gets $y_1^{\prime\prime -1}F^{-1}(x_1')^{-1}$ in the new coordinates. Thus in the new coordinates, $Y_{v,w} \subseteq {}^{\bU,F\bU'}\widehat\Sigma_w$ is the preimage under 
\[
y := y_1'\dot w \tau_1 z_1 F^{-1}(x_1')^{-1} \colon \widehat\Sigma_w \rar \bG
\]
of $\pr^{-1}(\bU_1 \dot v \bT_1'\bU_1') \subseteq \bG$. We have to show that for any $\obF$-algebra $R$ and any $(t,t') \in H_w(R)$, the map $(t,t') \colon Y_{v,w,R} \rar {}^{\bU,F\bU'}\widehat\Sigma_{w,R}$ factors through $Y_{v,w,R} \subseteq {}^{\bU,F\bU'}\widehat\Sigma_{w,R}$, that is $y \circ (t,t') \colon Y_{v,w,R} \rar \bG_R$ factors through the locally closed subset $\pr^{-1}(\bU_1 \dot v \bT_1'\bU_1')_R \subseteq \bG_R$. It suffices to do so on points. Let $(X_1,X_1',Y_1',T_1,Z_1,Y_1'') \in Y_{v,w}(R')$ for some $R$-algebra $R'$. Then 

\begin{align*}
y\circ(t,t')(X_1,X_1',Y_1',T_1,Z_1,Y_1'') &= y(F(t)X_1F(t)^{-1}, F(t')X_1'F(t')^{-1}, F(t)Y_1'F(t)^{-1},\dot w^{-1}t\dot w T_1 t^{\prime -1},t'Z_1t^{\prime -1},F(t')Y_1''F(t')^{-1}) \\
&= F(t)Y_1'F(t)^{-1} t\dot wT_1 Z_1 F^{-1}(X_1')^{-1} t^{\prime -1} \\
&= \underbrace{F(t)Y_1'F(t)^{-1}}_{\in \bU(R')} \cdot \underbrace{t}_{\in \bT(R')}  \cdot \underbrace{Y_1^{\prime -1}}_{\in \bU(R')} \cdot \underbrace{Y_1' \dot w T_1 Z_1 F^{-1}(X_1')^{-1}}_{\substack{\in \pr^{-1}(\bU_1\dot v \bT_1' \bU_1')(R') \\ \text{by assumption}}} \cdot \underbrace{t^{\prime -1}}_{\in \bT'(R')}.
\end{align*}
The last expression clearly lies in $\pr^{-1}(\bU_1\dot v \bT_1' \bU_1')(R')$ and we are done.
\end{proof}

Let us also look at the converse situation. The inverse of $\alpha$ is given by $(x_1,x_1',y_1) \mapsto (x_1, F^{-1}(x),y_1F^{-1}(x_1')^{-1})$. 

\begin{lm}\label{lm:alpha_inverse_image_stability_Lusztig}
For $v,w\in W(\bT_1,\bT_1')$ let $Z_{v,w} \subseteq {}^{\bU,\bU'}\widehat\Sigma_v$ be defined by the Cartesian diagram
\[
\xymatrix{
Z_{v,w} \ar[r] \ar[d] & {}^{\bU,\bU'}\widehat\Sigma_v \ar[d] \\
{}^{\bU,\bU'}\Sigma_v \cap \alpha^{-1}({}^{\bU,F\bU'}\Sigma_w) \ar[r] & {}^{\bU,\bU'}\Sigma_v
}
\]
where the left lower entry is the scheme-theoretic intersection inside ${}^{\bU,F\bU'}\Sigma_w$. Then $Z_{v,w}$ is stable under the action of $H_v'$ on ${}^{\bU,\bU'}\widehat\Sigma_v$ defined in Lemma \ref{lm:Lusztigs_extension_of_action}(ii).
\end{lm}
\begin{proof}
In terms of the presentation \eqref{eq:first_presentation_widehat_Sigma} of ${}^{\bU,\bU'}\widehat\Sigma_v$ where we denote the coordinates by $x,x',y',\tau,z,y''$, consider the morphism 
\[
y'\dot v \tau z y'' x' \colon \widehat\Sigma_v \rar \bG.
\]
Then $Z_{v,w}$ is the preimage under this morphism of $\pr^{-1}(\bU_1 \dot w \bT_1' F\bU_1')$. Now we make the change  of coordinates ($xF(y) \mapsto x$, $x'F(y'')^{-1} \mapsto x'$) from \eqref{eq:first_presentation_widehat_Sigma} to \eqref{eq:widehat_Sigma_after_change_var}. Then the expression $x'$ in the old coordinates becomes $x'F(y'')$ in the new coordinates. Hence in the new coordinates $Z_{v,w}$ is the preimage of $\pr^{-1}(\bU_1 \dot w \bT_1' F\bU_1')$ under 
\[
y_1 := y' \dot v \tau z y'' x'F(y'') \colon \widehat\Sigma_v \rar \bG.
\]
Let $R$ be an $\obF$-algebra and $(t,t') \in H_v'(R)$. As in the proof of Lemma \ref{lm:alpha_image_stable_under_Lusztigs_extension}, we have to show that $y_1 \circ (t,t') \colon Z_{v,w,R} \rar \bG_R$ factors through $\pr^{-1}(\bU_1 \dot w \bT_1' F\bU_1')_R \subseteq \bG_R$. Let $(X,X',Y',T,Z,Y'') \in Z_{v,w}(R')$ for some $R$-algebra $R'$. Then 
\begin{align*}
y_1 \circ (t,t')(X,X',Y',T,Z,Y'') &= y_1(tXt^{-1}, t'X't^{\prime -1}, tY't^{-1}, v^{-1}tvTt^{\prime -1}, t^{\prime} Z t^{\prime -1}, t' Y'' t^{\prime -1}) \\
&= t Y' \dot v T z Y'' X' t^{\prime -1}F(t') F(Y'') F(t')^{-1} \\
&= \underbrace{t}_{\in \bT(R')} \cdot \underbrace{Y' \dot v T z Y'' X' F(Y'')}_{\substack{\in \pr^{-1}(\bU_1 \dot w \bT_1' F\bU_1')(R') \\ \text{by assumption}}} \cdot \underbrace{F(Y'')^{-1}}_{\in F\bU'(R')} \cdot \underbrace{t^{\prime -1}}_{\in \bT'(R')} \cdot \underbrace{F(t') F(Y'') F(t')^{-1}}_{\in F\bU'(R')}
\end{align*}
The last expression lies in $\pr^{-1}(\bU_1 \dot w \bT_1' F\bU_1')(R')$ and we are done.
\end{proof}

\begin{rem}
There seem to be no analogs of these lemmas for $\widetilde\Sigma_v$ (from Section \ref{sec:another_extension_of_action}) instead of $\widehat\Sigma_v$.  
\end{rem}

\section{Regularity of certain subgroups}\label{sec:regularity}

The purpose of this section is to show that the groups $H_v,H_v', \dots $ produced in Section \ref{sec:extension_of_action} contain $\obF$-reductive subgroups under which the varieties $\widehat\Sigma_v$ and $\widetilde\Sigma_v$ have finitely many fixed points. This will be the key for computing their cohomology, as given in Theorem~\ref{thm:sigmav}. Note that this strategy was already used in  \cite[1.9(e)]{Lusztig_04}, but with much bigger versions of $H_v$.

\smallskip

Throughout this section we work in  the setup of Theorem \ref{thm:higher_mackey}. In particular, $\bfG$ is unramified, $\bf x$ is hyperspecial, and $(\bfT,\bfU)$, $(\bfT',\bfU')$ are Coxeter pairs with ${\bf x} = {\bf x}_{\bfT} = {\bf x}_{\bfT'}$. Thanks to the reduction results in Section~\ref{ssec:reduction-to-simple} we  will assume in addition that $\bfG$ is absolutely almost simple over $k$, i.e., that the Dynkin diagram of the split group $\bfG_{\breve k}$ is connected.

\subsection{Pull-back of a cocharacter under the Lang map}\label{sec:regularity_preparations}
We may identify the groups of cocharacters $X_\ast(\bT_1)$, $X_\ast(\bfT)$, and similarly for characters. The Frobenius $F$ acts on $X_\ast(\bT_1)$, $X^\ast(\bT_1)$ and these actions induce $\bQ$-linear automorphisms of the $\bQ$-vector spaces $X_\ast(\bT_1)_\bQ$, $X^\ast(\bT_1)_\bQ$. Let $\bfG^{\ad}$ be the adjoint quotient of $\bfG$ and $\bfT^{\ad}$ the image of $\bfT$ in $\bfG^{\rm ad}$, such that $X_{\ast}(\bfT^{\ad})$ is a quotient of $X_\ast(\bfT)$, and $X^\ast(\bfT^{\ad}) \subseteq X^\ast(\bfT)$.  

\smallskip

For $\chi \in X_\ast(\bT_1)$, we are interested in (the connected component of the) subgroup
\begin{equation}\label{eq:Lang_pullback_of_character}
H_\chi = \{ t \in \bT_1 \colon t^{-1}F(t) \in \im(\chi \colon \bG_m \rar \bT_1) \} \subseteq \bT_1.
\end{equation}

\begin{lm}\label{lm:connected_component}
Let $\chi \in X_\ast(\bT_1)$. 
There exists $0 \neq \mu \in X_\ast(\bT_1)$ such that  $F\mu - \mu \in \bQ\cdot\chi$. Such $\mu$ is
unique up to a scalar and we have $H_\chi^\circ = \im(\mu)$.
\end{lm}
\begin{proof}
By \cite[Prop. 13.7]{DigneM_91}, the map $F-1\colon X_\ast(\bT_1) \rar X_\ast(\bT_1)$ is injective and has finite cokernel. 
Therefore there exists $0 \neq \mu \in X_\ast(\bT_1)$, unique up to a scalar, such that $F\mu - \mu \in \bQ\cdot\chi$. This implies that $\im(\mu)$ is a one-dimensional subtorus of $\bT_1$ contained in $H_\chi$. Since $H_\chi$ is one-dimensional, this forces $\im(\mu) =H_\chi^\circ$. 
\end{proof}

Recall from \S\ref{ssec:condition} that $\{\alpha^\ast \colon \alpha \in \Delta\} \subseteq X_\ast(\bT_1^{\ad})_\bQ$ is the set of fundamental coweights, defined as the basis of $X_\ast(\bT_1^{\ad})_\bQ$ dual to $\Delta$.  

\begin{prop}\label{prop:existence-of-good-coxeter}
Assume Condition \eqref{cond:condition_on_rootsystem} holds for $q$ and $\bfG$. Then there exists a set of simple roots $\Delta \subset \Phi(\bfT,\bfG)$ such that
\begin{itemize}
 \item[(i)] $F$ acts on $X^\ast(\bfT)$ as $qc\sigma$ where $\sigma$ satisfies $\sigma(\Delta) = \Delta$, $c \in W$ and $c\sigma$ is a twisted Coxeter element of $(W,\sigma)$ such that 
 $$ \ell(c\sigma(c)\cdots \sigma^{i-1}(c))= i\ell(c)$$
 for all $0 \leq i \leq h/2$, where $h$ is the Coxeter number of $(W,\sigma)$.
 \item[(ii)] For all $\alpha \in \Delta$, and all $\gamma \in \Phi(\bfT,\bfG)$ we have $H_{\alpha^*}^\circ \not\subseteq \ker(\gamma)$. 
\end{itemize}
\end{prop}

\begin{proof}
Let $\tau = q^{-1}F$. Then the order of $\tau$ is $h$. Let $g \in \bfG$ be such that $\bfT_0 = {}^g \bfT$. By assumption on $\bfT$ the endomorphism ${}^g \tau$ of $X^\ast(\bfT_0)$ lies in the $W_0$-conjugacy class of twisted Coxeter elements (this class is unique by \cite[Thm. 7.6]{Springer_74}). Therefore the same holds for $\tau$ in $W$. Let  $\zeta =\mathrm{exp}(2\pi \mathrm{i}/h)$. By \cite[Thm. 7.6]{Springer_74}, the $\zeta$-eigenspace of $\tau$ on $X^\ast(\bfT^{\ad})_{\bC}$ is one-dimensional and is not contained in any reflection hyperplane. Let $0 \neq  v \in X^\ast(\bfT^{\ad})_\bC$ be an eigenvector of $\tau$ for the eigenvalue $\zeta$ such that $\mathrm{Re}(\langle v,\alpha^\vee\rangle) \neq 0$ for all $\alpha \in \Phi(\bfT,\bfG)$. Then as shown in the proof \cite[Prop. 4.10]{Springer_74} the condition  $\mathrm{Re}(\langle v,\alpha^\vee\rangle) > 0$ defines a set of positive roots $ \Phi^+ \subset \Phi(\bfT,\bfG)$, hence a basis $\Delta$. Let $c \in W$ be the unique element in $W$ such that $c(\Delta) = \tau(\Delta)$ and $\sigma = c^{-1} \tau$. Then $\sigma(\Delta) = \Delta$ and (i) follows from \cite[Prop. 6.5]{BroueMichel_94}.

\smallskip

Let $\alpha \in \Delta$ be a simple root and $\gamma \in \Phi(\bfT,\bfG)$ be any root. The orbit of $\gamma$ under $\tau = q^{-1} F = c\sigma $ has exactly $h$ elements, see \cite[Thm. 7.6]{Springer_74}. If $V = \langle \tau^i(\gamma) \colon i =0,\ldots,h-1\rangle$ is the $\bC$-vector subspace of $X^\ast(\bfT^{\ad})_\bC$ spanned by the orbit, then $\tau$ restricts to an automorphism of $V$ of order $h$. In particular it must contain the eigenvector $v$ defined above. Since $\alpha^*$ is a non-negative combination of simple coroots we deduce that $\mathrm{Re}(\langle v,\alpha^*\rangle) > 0$, which forces $\langle \tau^i(\gamma),\alpha^*\rangle \neq 0$ for some $i$. Let $i_0 \in \{0,\ldots,h-1\}$ be maximal such that $\langle \tau^i(\gamma),\alpha^*\rangle \neq 0$. Then
$$ \begin{aligned}
\sum_{i =0}^{h-1} \langle F^i(\gamma),\alpha^* \rangle & \, = \sum_{i =0}^{h-1} q^i \langle \tau^i(\gamma),\alpha^* \rangle \\
& \, = q^{i_0} \langle \tau^{i_0}(\gamma),\alpha^* \rangle + \sum_{i =0}^{i_0-1} q^i \langle \tau^i(\gamma),\alpha^* \rangle
\end{aligned}$$ 
so that
$$ \begin{aligned}
\left|\sum_{i =0}^{h-1} \langle F^i(\gamma),\alpha^* \rangle \right|& \, \geq 
 q^{i_0} | \langle \tau^{i_0}(\gamma),\alpha^* \rangle| - \sum_{i =0}^{i_0-1} q^i |\langle \tau^i(\gamma),\alpha^* \rangle|\\ 
 & \, \geq q^{i_0} - \sum_{i =0}^{i_0-1} q^iM =  q^{i_0} - M \frac{q^{i_0}-1}{q-1}\\
 & \,  \geq q^{i_0}- (q-1) \frac{q^{i_0}-1}{q-1} = 1
\end{aligned}$$ 
since by Condition \eqref{cond:condition_on_rootsystem}  we have $q-1 \geq M$. This proves that  $\sum_{i =0}^{h-1} \langle F^i(\gamma),\alpha^* \rangle \neq 0$. Now recall that $F^h = q^h$ on $X^*(\bfT)$. We have $(F-1)\sum_{i =0}^{h-1} F^i = F^h-1 = q^h-1$ therefore $(F-1)$ is invertible on $X^*(\bfT)$ and $(F-1)^{-1} = (q^h-1)^{-1} \sum_{i =0}^{h-1} F^i$. We deduce that 
$$\langle (F-1)^{-1}\gamma, \alpha^*\rangle  = \langle \gamma, (F-1)^{-1}\alpha^*\rangle   \neq 0.$$
Consequently, for any $0 \neq \mu \in (F-1)^{-1} \bQ \cdot \alpha^*$ we have $\langle \gamma,\mu\rangle \neq 0$. Using Lemma~\ref{lm:connected_component} we get $H_{\alpha^*}^\circ = \im(\mu)$ and we deduce that $H_{\alpha^*}^\circ \not\subseteq \ker(\gamma)$.
\end{proof}

\subsection{A consequence}\label{sec:consequences_regularity}

We have the short exact sequence of $\bF_q$-groups
\[
0 \rar \bT^1 \rar \bT \rar \bT_1 \rar 0,
\]
which is (canonically) split by the Teichm\"uller lift. Moreover, we have an isomorphism $\bT \cong \bT^1 \times \bT_1$ which sends the unipotent part $\bT_{{\rm unip}}$ to $\bT^1$ and the reductive part $\bT_{{\rm red}}$ to $\bT_1$. This also applies to $\bT'$ instead of $\bT$.

Let now $\bfL$ be a proper Levi subgroup of $\bfG$ containing $\bfT$, and let $v \in W(\bT_1, \bT'_1)$. We will be interested in the closed subgroup 
\[
H_{\bfL,v,r} = \{(t,t') \in \bT \times \bT' \colon t^{-1}F(t) = \dot v t^{\prime -1}F(t') \dot v^{-1} \text{ centralizes $\bL$} \} \subseteq \bT \times \bT'.
\]
Being affine and commutative, $H_{\bfL, v, r}$ decomposes into the product of its unipotent and reductive parts, $H_{\bfL, v, r} \cong H_{\bfL, v, r,{\rm unip}} \times H_{\bfL, v, r,{\rm red}}$, and we have $H_{\bfL, v, r,\red} \subseteq \bT_\red \times \bT'_{\rm red} \cong \bT_1 \times \bT_1'$. 

\begin{prop}\label{prop:regularity_general}
Assume Condition \eqref{cond:condition_on_rootsystem} holds for $q$ and $\bfG$. Suppose that $(\bfT,\bfU)$, $(\bfT',\bfU')$ are such that the corresponding sets of simple roots satisfy the conclusion of Proposition \ref{prop:existence-of-good-coxeter}.
Let $\bfL$, $v$ be as above. Consider the connected component $H_{\bfL,v,r, \red}^\circ$ of the reductive part of $H_{\bfL,v,r}$. Let $H$ (resp. $H'$) denote the image of $H_{\bfL,v,r, \red}^\circ$ under 
\[
H_{\bfL,v,r, \red}^\circ \har H_{\bfL,v,r} \rar \bT \times \bT' \tar \bT_1 \times \bT_1' \stackrel{\pr}{\tar} \bT_1,
\]
(resp. the image of the same map with $\bT_1$ on the right replaced by $\bT_1'$).
Then for all $\gamma \in \Phi(\bfT,\bfG)$, $H$ is not contained in the subtorus $\ker(\gamma) \subseteq \bT_1$, and similarly for $H'$ and all $\gamma' \in \Phi(\bfT',\bfG)$.
\end{prop}

\begin{proof}
Enlarging $\bfL$ makes its centralizer smaller, hence we may assume that $\bfL$ is a maximal proper Levi subgroup containing $\bfT$. We show only the claim for $H$, the one for $H'$ has a similar proof. Let
\[
H_{\bfL,r} = \{t \in \bT \colon t^{-1}F(t) \text{ centralizes } \bL \} \subseteq \bT.
\] 
The projection to the first factor $H_{\bfL,v,r} \rar H_{\bfL,r}$, $(t,t') \mapsto t$ is surjective (by Lang's theorem for the connected group $\bT'$), hence induces also a surjection on the reductive parts and hence also on their connected components, so it suffices to show that the connected component of
\[
H_1' := \im\left( H_{\bfL,r,\red}^\circ \har H_{\bfL,r} \rar \bT \tar \bT_1 \right)
\]
is not contained in $\ker(\gamma)$ for any $\gamma \in \Phi(\bfT,\bfG)$. 

By maximality of $\bfL$ there exists a system of simple positive roots $\Delta_1 \subseteq \Phi(\bfT,\bfG)$ and some $\alpha \in \Delta_1$ such that $\bfL$ is generated by $\bfT$ and all $\bfU_{\beta}$, $\bfU_{-\beta}$ with $\beta \in \Delta_1 \sm \{\alpha\}$. Alternatively, we can characterize $\bfL$ as follows: $\Delta_1$ forms a basis of $X^\ast(\bfT^{\rm ad})_\bQ$, and we have the fundamental coweights $\{\beta^\ast\}_{\beta \in \Delta_1}$ which form the dual basis of $X_\ast(\bfT^{\rm ad})_\bQ$. Then $\bfL$ is equal to the centralizer in $\bfG$ of a(ny) lift of $\alpha^\ast$ to $X_\ast(\bfT)_\bQ$ (again denoted $\alpha^\ast$). By Proposition \ref{prop:existence-of-good-coxeter}, the subgroup $H_{\alpha^\ast}^\circ$ of $\bT_1$ studied in Section \ref{sec:regularity_preparations} is not contained in $\ker(\gamma)$ for any $\gamma \in \Phi(\bfT,\bfG)$. Thus it suffices to show that $H_1' \supseteq H_{\alpha^\ast}^\circ$.

We have the Teichm\"uller lift ${\rm TM} \colon \bT_1 \rar \bT$, inducing an isomorphism $\bT_1 \stackrel{\sim}{\rar} \bT_\red$ onto the reductive part of $\bT$. Restricted to $H_{\alpha^\ast}^\circ$, ${\rm TM}$ induces an isomorphism ${\rm TM} \colon H_{\alpha^\ast}^\circ \stackrel{\sim}{\rar} {\rm TM}(H_{\alpha^\ast}^\circ)$ onto a subgroup of $\bT_\red$.

\begin{lm}\label{lm:TMlift_centralizes_Lr}
For any $\tilde{t} \in {\rm TM}(H_{\alpha^\ast}^\circ)$, $\tilde{t}^{-1} F(\tilde{t})$ centralizes $\bU_{\beta,r}$ for all $\beta \in \Phi(\bfT,\bfL)$ and consequently, it centralizes $\bL_r$. In particular, we have ${\rm TM}(H_{\alpha^\ast}^\circ) \subseteq H_{\bfL,r}$.
\end{lm}
\begin{proof}[Proof of Lemma \ref{lm:TMlift_centralizes_Lr}] Teichm\"uller lift commutes with Frobenius $F$, hence the map $t \mapsto t^{-1}F(t) \colon H_{\alpha^\ast}^\circ \rar \im(\alpha^\ast)$ induces a map $\tilde t \mapsto \tilde t^{-1} F(\tilde t) \colon {\rm TM}(H_{\alpha^\ast}^\circ) \rar {\rm TM}(\im(\alpha^\ast))$. Thus we have to show that ${\rm TM}(\im(\alpha^\ast)) \subseteq \bT$ centralizes $\bU_{\beta,r}$.

We have the homomorphism $\bT \rar \Aut(\bU_{\beta})$ given by the action of $\bT$ on $\bU_{\beta}$. The group $\bU_{\beta} = \bU_{\beta,r}$ comes with a filtration by closed subgroups $\bU_{\beta}^i = \ker(\bU_{\beta,r} \rar \bU_{\beta,i})$ ($0\leq i\leq r$) and the action of $\bT_r$ preserves this filtration, i.e., the above homomorphism factors through a homomorphism 
\[
\bT \rar \Aut_{\rm fil}(\bU_{\beta}), 
\]
where $\Aut_{\rm fil}(\bU_{\beta}) \subseteq \Aut(\bU_{\beta})$ is the subgroup of automorphisms preserving the filtration. This subgroup fits into an exact sequence
\[
1 \rar \Aut_{\rm fil, 0}(\bU_{\beta}) \rar \Aut_{\rm fil}(\bU_{\beta}) \rar Q \rar 1,
\]
where $\Aut_{\rm fil, 0}(\bU_{\beta})$ is the subgroup of automorphisms inducing the identity on the graded object ${\rm gr}^\bullet \bU_\beta = \bigoplus_{i=0}^{r-1} \bU_{\beta,i+1}^i$, and $Q$ is defined by exactness of the above sequence. The composition 
\[
{\rm TM}(\im(\alpha^\ast)) \subseteq \bT_r \rar \Aut_{\rm fil}(\bU_{\beta})
\]
factors through $\Aut_{\rm fil, 0}(\bU_{\beta})$: Indeed, the image of ${\rm TM}(\im(\alpha^\ast))$ in $\bT_1$ lies in $\im(\alpha^\ast) \subseteq \ker(\beta)$ (the latter inclusion holds as $\langle \beta, \alpha^\ast \rangle = 0$), hence it acts trivially on $\bU_{\beta,i+1}^i$ for each $0 \leq i \leq r-1$. But $\Aut_{\rm fil, 0}(\bU_{\beta})$ is unipotent, whereas $TM(\im(\alpha^\ast)) \cong \im(\alpha^\ast)$ is a torus, hence the resulting morphism 
\[
{\rm TM}(\im(\alpha^\ast)) \rar \Aut_{\rm fil,0}(\bU_{\beta})
\]
is trivial. This proves the lemma.
\end{proof}

By Lemma \ref{lm:TMlift_centralizes_Lr}, ${\rm TM}(H_{\alpha^\ast}^\circ) \subseteq H_{\bfL,r}$. Being reductive and connected, ${\rm TM}(H_{\alpha^\ast}^\circ)$ is thus contained in $H_{\bfL,r,\red}^\circ$. This shows that the image of $H_{\bfL,r,\red}^\circ$ in $\bT_1$ contains the image of ${\rm TM}(H_{\alpha^\ast}^\circ)$, which is just $H_{\alpha^\ast}^\circ$. \qedhere
\end{proof}

\begin{cor}\label{cor:regularity_general}
Under the assumptions of Proposition \ref{prop:regularity_general}, let $\widetilde H$ (resp. $\widetilde H'$) denote the image of the map 
\[
H_{\bfL,v,r, \red}^\circ \har H_{\bfL,v,r} \rar \bT \times \bT' \stackrel{\pr}{\tar} \bT
\]
(resp. the image of the same map with $\bT$ on the right replaced by $\bT'$).
Let $\bV$ be the subgroup of $\bG$ corresponding to the unipotent radical $\bfV$ of an arbitrary Borel subgroup of $\bfG$ containing $\bfT$. Then $\bV^{\widetilde H} = \{1\}$, i.e., the only element of $\bV$ fixed by the adjoint action of $\widetilde H$ is $1$. The analogous statement holds for $\bT',\bV',\widetilde H'$.
\end{cor}

\begin{proof}
We prove only the first claim. The proof of the second is similar. Any element of $\bV(\obF)$ has a unique presentation as a product of elements in the  subgroups $\bU_\gamma$ corresponding to root subgroups $\bfU_{\gamma} \subseteq \bfG$ for $\gamma \in \Phi(\bfT,\bfV)$, and this product decomposition is compatible with the adjoint action of $\bT$. This reduces the corollary to the claim that $\bU_\gamma^{\widetilde H} = \{1\}$ for all $\gamma \in \Phi(\bfT,\bfV)$. For the latter, we can use induction on $1\leq r'\leq r$: it suffices to show that if $x \in \bU_{\gamma,r}^{\widetilde H_i}$ and $x$ projects to $1$ under $\bU_{\gamma,r} \rar \bU_{\gamma,r' - 1}$, then it projects to $1$ under $\bU_{\gamma,r} \rar \bU_{\gamma,r'}$. The adjoint action of $\bT_1$ on $\bU_{\gamma,r'}^{r'-1}$ can be described as follows: fix an isomorphism $\bfG_{a,\breve k} \stackrel{\sim}{\rar} \bfU_{\gamma}$, which is part of an \'epinglage for $\bfG$. It induces an isomorphism $u_{\gamma,r'}^{r-1} \colon \bG_{a,\obF} \stackrel{\sim}{\rar} \bU_{\gamma,r'}^{r'-1}$, and the adjoint action is given by ${\rm Ad}(t)(u_{\gamma,r'}^{r'-1}(x)) = u_{\gamma,r'}^{r'-1}(\gamma(t)x)$. Now the result follows, as the image of $\widetilde H$ in $\bT_1$ is not contained in $\ker(\gamma)$ by Proposition \ref{prop:regularity_general}.
\end{proof}

\section{Cohomology of $\Sigma$}\label{sec:proof_for_aasG}

As in Section~\ref{sec:regularity} we assume that $\bfG$ is an unramified absolutely almost simple group. 
Throughout this section we will assume that $\bfG$ is an unramified absolutely almost simple group, and that condition \eqref{cond:condition_on_rootsystem} holds for $\bfG$ and $q$. We fix a Coxeter pair $(\bfT,\bfU)$ as in Proposition~\ref{prop:existence-of-good-coxeter}. In particular the action of $F$ on $W$ is given by $F = q c\sigma$ where $\sigma$ is an automorphism of $W$ permuting the simple reflections and $c\sigma$ is a twisted Coxeter element of $(W,\sigma)$. The purpose of this section is to show that \eqref{eq:vTF_and_Sigma_v} holds for such a Coxeter pair. This will imply Theorem~\ref{thm:higher_mackey} for general unramified groups.

\subsection{Non-emptyness of cells}
We give here conditions for a cell ${}^{\bU,\bU'}\Sigma_v$ to be empty (see \eqref{eq:Sigma_v_def} for the definition of the cell). Unlike Theorem~\ref{thm:sigmav} which is stated in the case where $\bU = \bU'$, we will work here with more general Coxeter pairs. 

\begin{prop}\label{prop:sigma_non_empty}
Let $a \in \mathbb{Z}$ and set $\bfU' :=F^a(\bfU)$. If $v \in W$ is such that ${}^{\bU,\bU'} \Sigma_v \neq \varnothing$, then at least one of the following holds
\begin{itemize}
\item[(i)] $v \in W^F$;
\item[(ii)] ${}^{v^{-1}}\bfU \cap \bfU'$ is contained in a proper Levi subgroup of $\bfG$ containing $\bfT$.
\end{itemize}
\end{prop}

\begin{proof}
Recall from Proposition \ref{prop:existence-of-good-coxeter} that $F = qc\sigma$ with $c\sigma$ a twisted Coxeter element. 
 Given $v \in W$, let us consider the condition
\begin{equation}\label{eq:non-empty}
v^{-1} \bB_1 F(\bB_1 v) \cap \bB'_1 F(\bB'_1) \neq \varnothing
\end{equation}
From the definition of ${}^{\bU,\bU'} \Sigma_v$ we see that if ${}^{\bU,\bU'} \Sigma_v$ is non-empty then there exists $x,x \in F(\bB_1)$ and $y \in \bB_1 v \bB_1'$ such that $xF(y) = yx'$. Writing $y = b v b'$ with $b \in \bB_1$ and $b' \in \bB_1'$ we deduce that $v^{-1} b^{-1} x F(bv) = b'vxF(b')^{-1}$ so that  
\eqref{eq:non-empty} holds. Therefore it is enough to show that if 
\eqref{eq:non-empty} holds for $v$ then (i) or (ii) hold as well. 

\smallskip Since $F(\bB) = {}^c \bB$ we have  $\bB'_1 = F^a(\bB_1) = {}^{d} \bB_1$ with $d = c\sigma(c) \cdots \sigma^{a-1}(c) = (c\sigma)^a \sigma^{-a}$. Using the fact that $d^{-1} F(d) = \sigma^a(c) c^{-1}$ we get
$$ \begin{aligned}
v^{-1} \bB_1 F(\bB_1 v) \cap \bB'_1 F(\bB'_1) & \, = \big(v^{-1} \bB_1 {}^c\bB_1 F(v) \big)\cap \big(d \bB_1 d^{-1} F(d) {}^c\bB_1 F(d^{-1})\big) \\
& \, =   \big(v^{-1} \bB_1 c\bB_1\sigma(v) c^{-1} \big)\cap \big(d \bB_1 \sigma^a(c) \bB_1 \sigma(d^{-1}) c^{-1}\big)\\
& \, =  v^{-1} \Big(\big(\bB_1 c \bB_1 \sigma(vd) \big) \cap \big(vd \bB_1 \sigma^a(c) \bB_1\big)\Big) \sigma(d^{-1}) c^{-1}.\end{aligned}$$ 
Therefore \eqref{eq:non-empty} is equivalent to 
$$\big(\bB_1 c \bB_1 \sigma(vd) \big) \cap \big(vd \bB_1 \sigma^a(c) \bB_1\big) \neq \varnothing$$
which in turn is equivalent to 
\begin{equation}\label{eq:double-bruhat}
 (\bB_1 c \bB_1 \sigma(vd) \bB_1) \cap (\bB_1 vd \bB_1 \sigma^a(c)  \bB_1) \neq \varnothing.
\end{equation}

\smallskip

Let $\Delta \subseteq \Phi(\bfT,\bfG)$ be the set of simple roots corresponding to $\bfU$.  Since $c\sigma$ is a twisted Coxeter element, there exists representatives of $\sigma$-orbits of simple reflections $s_1,\ldots,s_r$ with $r = |\Delta/\sigma|$ such that  
$c =s_1 s_2 \cdots s_r$. Given $I \subset \{1,\ldots,r\}$ we will denote by $W_I$ the smallest $\sigma$-stable parabolic subgroup of $W$ containing $s_i$ for all $i\in I$  and by $c_I = \prod_{i \in I} s_i$  the element of $W_I$ obtained from $c$ by keeping the simple reflections labelled by $I$. Note that $c_I \sigma$ is a twisted Coxeter element of $(W_I,\sigma)$. 

\smallskip

Assume that \eqref{eq:double-bruhat} holds. Since $c$ contains each simple reflection at most once, the Bruhat cells $\bB_1 u \bB_1$ contained in $\bB_1 c \bB_1 \sigma(vd) \bB_1$ (resp. in $\bB_1 vd \bB_1 \sigma^a(c) \bB_1$) are attached
to elements $u \in W$ of the form $u = c_I \sigma(vd)$ for some $I \subset \{1,\ldots,r\}$ (resp. $u = vd\sigma^a(c_J) $ for some $J \subset \{1,\ldots,r\}$). Consequently if \eqref{eq:double-bruhat} holds then there exists $I,J \subset  \{1,\ldots,r\}$ such that $c_I \sigma(vd) = vd\sigma^a(c_J)$. Set $w := w_0 vd$ where $w_0$ is the longest element of $W$. Since $\sigma(w_0) = w_0$ we have $({}^{w_0} c_I) \sigma(w) = w \sigma^a(c_J)$. Let $K \subset \{1,\ldots,r\}$ be such that $W_K := {}^{w_0} W_I$. Then  ${}^{w_0} c_I \sigma$ is a twisted Coxeter element of $W_K$ (but not necessarily equal to $c_K \sigma $).  
By \cite[Prop. 2.1.7]{GeckP_00}, one can write $w = w_1 x w_2$ where $x \in W$ is $K$-reduced-$J$ (i.e. of minimal length in $W_K x W_I$) and $(w_1,w_2) \in W_K\times W_J$. Since $x$ is $K$-reduced-$J$ we claim that
$$ W_{J} \cap x^{-1} W_K \sigma(x) = \left\{\begin{array}{ll} W_{J \cap K^x} & \text{if } \sigma(x) = x \\
\emptyset & \text{otherwise}.  
\end{array} \right.$$
The proof of this claim follows for example from the proof of \cite[Thm. 2.1.12]{GeckP_00}. Indeed, if $W_{J} \cap x^{-1} W_K \sigma(x)$ is non-empty then there exists $y \in W_J$ and $z \in W_K$ such that $xy = z\sigma(x)$. Since 
$x$ is reduced-$J$ and $\sigma(x)$ is $K$-reduced we have necessarily $\ell(y) = \ell(z)$. Let $y = y_1 \cdots y_m$ be a reduced expression of $y$. We define inductively $z_i \in W_K$ and $x_i$ a $K$-reduced element by the conditions $x_0 = x$ and $x_{i-1} y_i = z_i x_i$ for all $i = 1,\ldots,m$. In particular $z = z_1 \cdots z_m$ and $x_m = \sigma(x)$. By Deodhar's Lemma \cite[Lem. 2.1.2]{GeckP_00} we have $\ell(z_i) = 1$ and $x_i = x_{i-1}$ for all $i$ (the case $z_i = 1$ does not happen since $\ell(z)=\ell(y) = m$). In particular $\sigma(x) = x$ and the result of \cite[Thm. 2.1.12]{GeckP_00} applies.

\smallskip

The equality $({}^{w_0} c_I) \sigma(w) = w \sigma^a(c_J)$ forces $W_{J} \cap x^{-1} W_K \sigma(x)$ to be non-empty, therefore $\sigma(x) = x$. Now the element 
$$ w_2^{-1} \sigma^a(c_J) \sigma(w_2) = x^{-1} w_1^{-1} ({}^{w_0} c_I) \sigma(w_1 x) = x^{-1} (w_1^{-1} w_0 c_I w_0^{-1} \sigma(w_1)) x$$
lies in $W_{J} \cap x^{-1}(W_K)x$ and is $\sigma$-conjugate to a twisted Coxeter element of $W_J$. Since Coxeter elements are elliptic, this forces  $W_{J \cap K^x} = W_J$, hence $J \subset K^x $. Similarly, we find $K \subset {}^xJ$ hence ${}^x J = K$. In particular one can write $w = w' x$ with $w' \in W_K$.

\smallskip

Let us now look more precisely at what elements $u \in W$ can appear. If $\bB_1 v d \sigma^a(c_J) \bB_1 \subset \bB_1 vd \bB_1 \sigma^a(c) \bB_1$ with $J = \{ j_1 < j_2 < \cdots < j_m\}$ then for all $i =0, \ldots,m$ and all $j_i < l < j_{i+1}$ we must have $v d \sigma^a(s_{j_1} \cdots s_{j_i} s_l) < v d \sigma^a(s_{j_1} \cdots s_{j_i})$, with the convention that $j_0 = 0$, $j_{m+1} = r+1$ and $s_{j_0}=1$. On the other hand, $w = w' x$ with $w' \in W_K$ and $x$ is $K$-reduced. Since ${}^x J= K$ we can write $K = \{k_1,\ldots,k_m\}$ with $x s_{j_i}x^{-1} = s_{k_i}$.
Then the condition
$$ w \sigma^a(s_{j_1} \cdots s_{j_i} s_l)  > w \sigma^a(s_{j_1} \cdots s_{j_i})$$
can be written
$$w' \sigma^a(s_{k_1} \cdots s_{k_i} x s_l)  > w' \sigma^a(s_{k_1} \cdots s_{k_i} x) $$
Now, since $w' \sigma^a(s_{k_1} \cdots s_{k_i}) \in W_K$ and ${}^x s_l \notin W_K$ this forces $x \sigma^a(s_l) > x$ hence $x s_l > x$ (recall that $\sigma(x) =x$). Indeed, if $\alpha_l$ denotes the simple root associated to $s_l$ then $w' \sigma^a(s_{k_1} \cdots s_{k_i}x) (\alpha_l) > 0$ by assumption. Since  $x(\alpha_l)$ is not in $\Phi_K$, the root subsystem associated to $W_K$, the element $w' \sigma^a(s_{k_1} \cdots s_{k_i}) \in W_K$ cannot change the sign of $x(\alpha_l)$, therefore $x(\alpha_l)>0$. Since $l$ runs over all the elements in $\{1,\ldots,r\}\smallsetminus J$ and $x$ is reduced-$J$ this proves that $x s > x$ for all simple reflections $s$ in $I$, and therefore $x =1$ since $x$ is $\sigma$-stable. Consequently $v d = w_0 w = w_0 w' \in w_0 W_K = W_J w_0$.  If $J \neq \{1,\ldots,r\}$, then ${}^{d^{-1}v^{-1}}\bfU \cap \bfU = {}^{d^{-1}}({}^{v^{-1}} \bfU \cap \bfU')$ is contained in the Levi subgroup of $\bfG$ corresponding to $W_J$, hence (ii) holds. Otherwise $I = J = K = \{1,\ldots,r\}$ and the relation $c_I \sigma(vd) = vdc_J$ is just $c\sigma(vd)=vd \sigma^a(c)$ which, with $d = (c\sigma)^a \sigma^{-a}$ gives $v \in W^{c\sigma} = W^F$, hence (i) holds.
\end{proof}

\subsection{Comparison of various cells}\label{sec:comparison_of_cells}

In this section we prove that \eqref{eq:vTF_and_Sigma_v} holds for the Coxeter pair $(\bfT,\bfU)$.  
Note that the Coxeter number $h$ (the order of $c\sigma$) is even unless $W$ is of type $A_n$ with $n$ even.
Proposition \ref{prop:existence-of-good-coxeter} implies that when $h$ is even we have $c \sigma(c) \cdots \sigma^{h/2-1}(c) = w_0$, the longest element in $W$.

\begin{lm}\label{lm:dim_sigma_coh_v_not_centralizing_c}
Assume that $v \in W \smallsetminus W^F$. Then 
$$ H_c^\ast({}^{\bU,\bU}\Sigma_v) = 0$$
as a virtual $(\bT\times^\bZ \bT)^F$-module.
\end{lm}
\begin{proof}
If ${}^{\bU,\bU}\Sigma_v$ is empty then the statement is trivial. Otherwise, Proposition~\ref{prop:sigma_non_empty} ensures that ${}^{v^{-1}}\bfU \cap \bfU$ is contained in a proper Levi subgroup $\bfL$ of $\bfG$ containing $\bfT$. In particular the torus $\bH := H_{\bfL,v,r, \red}^\circ$ defined in \S\ref{sec:consequences_regularity} is contained in $H_v$, which by  \ref{lm:Lusztigs_extension_of_action} acts on $\widehat\Sigma_v$. Using Corollary~\ref{cor:regularity_general}, we see that $({}^{\bU,\bU}\widehat\Sigma_v)^{\bH} $ is empty since $F(v) \neq v$. By \eqref{eq:Euler_char}, this shows that $H_c^\ast({}^{\bU,\bU}\widehat\Sigma_v) = 0$. The same holds for ${}^{\bU,\bU}\Sigma_v$ since 
it is related to ${}^{\bU,\bU}\widehat \Sigma_v$ by a $(\bT\times^\bZ \bT)^F$-equivariant map $\widehat\Sigma_v \rar \Sigma_v$ which is a Zariski-locally trivial fibration with fibers isomorphic to the perfection of a fixed affine space.
\end{proof}

Given $a \in \mathbb{Z}$ and $v \in W^F$ we define the virtual $(\bT\times^\bZ \bT)^F$-modules
\begin{align*}
h_{a,v} &:=  H_c^\ast({}^{\bU,F^a(\bU)}\Sigma_{v},\cool)  = H_c^\ast({}^{\bU,F^a(\bU)}\widehat\Sigma_{v},\cool)  = H_c^\ast({}^{\bU,F^a(\bU)}\widetilde\Sigma_{v},\cool), \\
\widetilde h_v &:= H_c^0((\dot v\bT)^F,\cool) = \cool(\dot v\bT)^F.
\end{align*}
Note that $h_{a,v}$ depends only on the class of $a$ modulo $h$, the Coxeter number.   

\begin{lm}\label{lm:known_cells}
Let $c_a = (c\sigma)^a\sigma^{-a}$. Assume that either
\begin{itemize} 
\item $h$ is even and $vc_a \in \{w_0,cw_0,w_0 \sigma^a(c^{-1})\}$; or 
\item $h$ is odd and $vc_a = c_{\lfloor h/2 \rfloor \pm1}$.
\end{itemize}
Then $h_{a,v} = \widetilde h_v$.
\end{lm}
\begin{proof}
Assume first that $h$ is even and that $vc_a=w_0$. Then ${}^{v^{-1}} \bU \cap F^a(\bU) = {}^{v^{-1}}\bU \cap {}^{c_a}\bU =   {}^{c_a}({}^{(vc_a)^{-1}}\bU \cap \bU) =  {}^{c_a}(\bU^- \cap \bU) = 1$ implies that $H_v = \bT\times\bT'$, which acts on $\widehat \Sigma_v$ by Lemma \ref{lm:Lusztigs_extension_of_action}. By Corollary \ref{cor:regularity_general} applied to the Levi subgroup $\bfL = \bfT$, the map
$$ \dot v \tau \in (\dot v \bT)^F \longmapsto (1,1,1,\tau,1,1) \in \widehat \Sigma_v$$
induces a $(\bT\times^\bZ \bT)^F$-equivariant isomorphism
$$ (\dot v \bT)^F \simeq (\widehat \Sigma_v)^{(H_v)_\red}$$
and the result follows. Similarly, the other two cases follow by using Lemma \ref{lm:new_simple_extensions}(i) resp. \ref{lm:new_simple_extensions}(ii) instead of Lemma \ref{lm:Lusztigs_extension_of_action}.
\smallskip

When $h$ is odd then $W$ is of type $A_n$ with $n$ even and $\sigma= 1$. In that case $h = n+1$ and $\ell(c) = n$. By Proposition \ref{prop:existence-of-good-coxeter} we have $\ell(c^{n/2}) = n^2/2$ and $\ell(c^{n/2+1}) = \ell(c^{-n/2}) =  n^2/2$.
Therefore if $k = \lfloor h/2 \rfloor \pm1$ we have $\ell(w_0 c_k) = n(n+1)/2-n^2/2 = n/2 < n$, which forces $w_0 c_k$ to lie in a proper parabolic subgroup of $W$. Consequently ${}^{v^{-1}} \bU \cap F^a(\bU) = {}^{v^{-1}}\bU \cap {}^{c_a}\bU =   {}^{c_a}({}^{(vc_a)^{-1}}\bU \cap \bU)  = {}^{c_a}({}^{(w_0c_k)^{-1}}\bU^- \cap \bU)$ lies in a proper parabolic subgroup of $\bfG$ containing $\bfT$. The result follows again by the combination of Lemma \ref{lm:Lusztigs_extension_of_action} and Corollary~\ref{cor:regularity_general}. 
\end{proof}

The key observation is the following proposition.

\begin{prop}\label{prop:compare_Coxeter_cells}
Let $a \in \bZ/h\bZ$ and $v \in W^F$. We have $h_{a,v} = h_{a+1,v}$, unless $\sigma$ is trivial and $v = w_0 c^{-a}$ or $v= w_0 c^{-a-1}$.
\end{prop}
\begin{proof}
As in Section \ref{sec:isomorphism} we have the isomorphism $\alpha \colon {}^{\bU,F^a(\bU)}\Sigma \rar {}^{\bU,F^{a+1}(\bU)} \Sigma$, and the cell ${}^{\bU,F^a(\bU)}\Sigma_{v}$ decomposes into finitely many locally closed $(\bT\times^\bZ \bT)^F$-stable pieces: 
\begin{equation}\label{eq:decomposition_of_coh_via_alpha}
{}^{\bU,F^{a}(\bU)}\Sigma_{v} = \bigcup_{w \in W} \alpha^{-1}\left(\alpha \left({}^{\bU,F^{a}(\bU)}\Sigma_{v}\right) \cap {}^{\bU,F^{a+1}(\bU)} \Sigma_w\right),
\end{equation}
As in Lemma \ref{lm:alpha_image_stable_under_Lusztigs_extension}, we have the $(\bT\times^\bZ \bT)^F$-stable piece ${}^a Y_{v,w} := Y_{v,w} \subseteq {}^{\bU,F^{a+1}(\bU)}\widehat\Sigma_{v}$, and it satisfies
\[
H_c^\ast({}^a Y_{v,w})  = H_c^\ast \left(\alpha\left({}^{\bU,F^{a}(\bU)}\Sigma_{v}\right) \cap {}^{\bU,F^{a+1}(\bU)} \Sigma_w \right)
\]
This and \eqref{eq:decomposition_of_coh_via_alpha} give
\begin{equation} \label{eq:hak_through_ha1k} 
h_{a,v} = \sum_{w\in W}H_c^\ast({}^aY_{v,w},\cool) 
\end{equation}
By Lemma \ref{lm:alpha_image_stable_under_Lusztigs_extension}, ${}^aY_{v,w} \subseteq {}^{\bU,F^{a+1}(\bU)}\widehat\Sigma_w$ is stable under the $H_w$-action on ${}^{\bU,F^{a+1}(\bU)}\widehat\Sigma_w$ as in Lemma \ref{lm:Lusztigs_extension_of_action}(i). If ${}^{w^{-1}} \bfU \cap F^{a+1}(\bfU)$
is contained in a proper Levi subgroup of $\bfG$ containing $\bfT$, then using again the argument as in the proof of Lemma~\ref{lm:dim_sigma_coh_v_not_centralizing_c} we have $H_c^\ast({}^aY_{v,w},\cool) = 0$ whenever $w \notin W^F$ (cf. Section \ref{sec:perfect_schemes_and_coh}). Consequently by Proposition~\ref{prop:sigma_non_empty} applied to $\bU' = F^{a+1}(\bU)$ we only need to consider the case where $w \in W^F$, so that
\begin{equation}\label{eq:hak_through_ha1k_simplified} 
h_{a,v} = \sum_{w \in W^F} H_c^\ast({}^aY_{v,w},\cool)
\end{equation}
Analogously one can decompose the cell ${}^{\bU,F^{a+1}(\bU)}\Sigma_{v}$ into finitely many locally closed $(\bT\times^\bZ \bT)^F$-stable pieces as follows:
$${}^{\bU,F^{a+1}(\bU)}\Sigma_{v} = \bigcup_{w \in W} \alpha \left({}^{\bU,F^{a}(\bU)}\Sigma_{w}\right) \cap {}^{\bU,F^{a+1}(\bU)} \Sigma_v,$$
and using Lemmas \ref{lm:alpha_inverse_image_stability_Lusztig} and \ref{lm:Lusztigs_extension_of_action}(ii) instead of Lemmas \ref{lm:alpha_image_stable_under_Lusztigs_extension} and \ref{lm:Lusztigs_extension_of_action}(i) we show
\begin{equation}\label{eq:hak_through_ha1k_simplified_2} 
h_{a+1,v} = \sum_{w \in W^F} H_c^\ast({}^aY_{w,v},\cool). 
\end{equation}

\begin{lm}\label{lm:emptyness_Yckck}
Let $v,w \in W^F$.
Assume that ${}^aY_{v,w} \neq \varnothing$ and $v\neq w$. Then $\sigma$ is trivial, $v = w_0c^{-a}$ and $v=wc$.
\end{lm}
\begin{proof}
The scheme ${}^aY_{v,w}$ can only be non-empty if $\alpha \left({}^{\bU,F^a(\bU)}\Sigma_{v}\right) \cap {}^{\bU,F^{a+1}(\bU)} \Sigma_{w} \neq \varnothing$. If this is the case, there must exist a point $(x,x',y) \in {}^{\bU,F^a(\bU)}\Sigma_{v}$, such that $\alpha(x,x',y) = (x,F(x'),yx') \in {}^{\bU,F^{a+1}(\bU)} \Sigma_{w}$. Let $y_1 = yx'$, and let $\bar{x}',\bar{y},\bar{y}_1$ denote the images of $x',y,   y_1$ in $\bG_1$. Write $\bB_1 = \bT_1\bU_1$. Given $k \in \mathbb{Z}$ we write $c_k = (c\sigma)^k \sigma^{-k}$ so that $F^k(\bB) = {}^{c_k} \bB$. We then have 
\begin{align}
\nonumber \bar y &\in \bB_1 v F^a(\bB_1) = \bB_1 v c_a \bB_1 c_a^{-1}, \quad  \bar{x}' \in F^{a+1}(\bU_1) = {}^{c_{a+1}}\bU_1, \quad \text{and} \\ 
\label{eq:yxprime_in_coset}\bar{y}_1 &\in \bB_1w F^{a+1}(\bB_1) = \bB_1 wc_{a+1}\bB_1 c_{a+1}^{-1}
\end{align}
From the latter two of these three conditions it follows that $\bar{y} = \bar{y}_1\bar{x}^{\prime -1} \in \bB_1 wc_{a+1} \bB_1 c_{a+1}^{-1}$, and we deduce from the first condition in \eqref{eq:yxprime_in_coset} that $ \bB_1 wc_{a+1} \bB_1 c_{a+1}^{-1} \cap  \bB_1 v c_a \bB_1 c_a^{-1}$ contains $\bar{y}$, hence is non-empty. Multiplying by $c^a$ from the right and using that $c_{a+1} = c_a \sigma^a(c)$ we get
\begin{equation}
\big(\bB_1 wc_{a+1} \bB_1 \sigma^{a}(c^{-1}) \big) \cap  \big( \bB_1 v c_a \bB_1\big)  \neq \emptyset.
\end{equation}
By \cite[Thm. 7.6(v)]{Springer_74}, there exists $k,l \in \{0,1,\ldots,h-1\}$ such that $\sigma^k = \sigma^l = 1$ and
$v = c_k$, $w = c_l$. Therefore the previous equation can be written
$$\big(\bB_1 c_{l+a+1} \bB_1 \sigma^{a}(c^{-1}) \big) \cap  \big( \bB_1 c_{k+a} \bB_1\big)  \neq \emptyset.$$
This implies that 
\begin{equation}\label{eq:double_coset_Coxeter_powers}
\big(\bB_1 c_{l+a+1} \bB_1 \sigma^{a}(c^{-1})\bB_1\big) \cap \big(\bB_1 c_{k+a} \bB_1\big) \neq \varnothing \quad \text{and}
\quad \big(\bB_1 c_{l+a+1} \bB_1\big) \cap \bB_1 c_{k+a}  \bB_1  \sigma^{a}(c) \bB_1\big) \neq \varnothing.
\end{equation}
As in the proof of Proposition~\ref{prop:sigma_non_empty}, recall that the elements $u \in W$ such that $\bB_1 u \bB_1 \subset \bB_1 c_{k+a} \bB_1  \sigma^{a}(c) \bB_1$ are of the form $c_{k+a} \sigma^{a}(c_I)$, where $c_I$ is obtained by removing some simple reflections in $c$. Therefore by \eqref{eq:double_coset_Coxeter_powers} we have $c_{k+a} \sigma^{a}(c_I) = c_{l+a+1}$ for some $c_I \leq c$, yielding $k \in \{l,l+1\}$. In addition when $\ell(c_{k+a+1}) = \ell(c) + \ell(c_{k+a})$ (or when  $\ell(c_{l+a}) = \ell(c^{-1}) + \ell(c_{l+a+1})$) only $c_I= c$ can appear, in which case $k =l$ and hence $v=w$, which contradicts the assumptions of the Lemma. Therefore we have $k = l+1$, $\ell(c_{k+a+1}) \neq \ell(c) + \ell(c_{k+a})$ and $\ell(c_{k+a-1}) \neq \ell(c^{-1}) + \ell(c_{k+a})$. Consequently $\ell(c_{k+a}) > \ell(c_{k+a\pm1})$ therefore $c_{k+a} = w_0$.
Note that we also have $\sigma = \sigma^{k} \sigma^{-l} = 1$ and the lemma follows.   
\end{proof}
Now we finish the proof of Proposition \ref{prop:compare_Coxeter_cells}. Applying Lemma \ref{lm:emptyness_Yckck} we deduce from equation \eqref{eq:hak_through_ha1k_simplified} that $h_{a,v} = H_c^\ast({}^aY_{v,v},\cool) $ unless $\sigma = 1$ and $v=w_0 c^{-a}$.
Similarly, equation \eqref{eq:hak_through_ha1k_simplified_2} yields $h_{a+1,v} =   H_c^\ast({}^aY_{v,v},\cool) $ unless $\sigma = 1$ and $vc=w_0 c^{-a}$. Therefore if $\sigma \neq 1$ or if $v \notin \{w_0 c^{-a}, w_0 c^{-a-1}\}$ we have
\begin{equation}\label{eq:comparison_hak_hak1}
h_{a,v} =   H_c^\ast({}^aY_{v,v},\cool)  = h_{a+1,v},
\end{equation}
which finishes the proof.
\end{proof}

\begin{proof}[Proof of Theorem \ref{thm:sigmav}]
By Lemma \ref{lm:dim_sigma_coh_v_not_centralizing_c} it suffices to show that $h_{0,v} = \widetilde h_v$ for all $v \in W^F$. 
Recall that by  \cite[Thm. 7.6(v)]{Springer_74} the elements in $W^F$ are of the form $c_k = (c\sigma)^k \sigma^{-k}$ for some $k \in \bZ$ with $\sigma^k =1$. 

\smallskip 

If $\sigma$ is non-trivial, then there exists $a \in  \bZ$ such that $vc_a = w_0$ (for example $a = h/2 - k$). By Lemma \ref{lm:known_cells} we have $h_{a,v} = \widetilde h_v$ and from Proposition \ref{prop:compare_Coxeter_cells} we deduce that $h_{0,v} = h_{a,v} = \widetilde h_v$.

\smallskip
If $\sigma = 1$, then $v = c^k$. Without loss of generality we can assume that $v \neq w_0$ (equivalently $k \neq h/2$) since in that case Lemma \ref{lm:known_cells} applies. Assume first that $0 \leq k < h/2$. Then $vc^a = c^{k+a} \neq w_0$ for all $0 \leq a < h/2-k$. Therefore by  Proposition \ref{prop:compare_Coxeter_cells} we have
$h_{0,v} = h_{a,v}$ in that case. If $h$ is even then $vc^{h/2-k-1} = w_0 c^{-1}$ in which case $h_{h/2-k-1,v}$ equals $\widetilde h_v$ by Lemma \ref{lm:known_cells}. If $h$ is odd then $h_{\lfloor h/2\rfloor-k,v}$ equals $\widetilde h_v$ by Lemma \ref{lm:known_cells} again. When $h/2 < k< h$ we have $h_{0,v} = h_{-a,v}$ for all $0 \leq a < h/2+k$, and a similar argument applies. 
\end{proof}

\bibliography{bib_ADLV}{}
\bibliographystyle{alpha}
\end{document}